\documentclass{amsart}

\usepackage{amssymb}
\usepackage{amsthm}
\usepackage{amsbsy}
\usepackage{subfigure}
\usepackage{amsmath,mathrsfs}
\usepackage{graphicx}

\usepackage{comment}
\usepackage[usenames,dvipsnames]{xcolor}
\usepackage[margin = 1.1in]{geometry}
\usepackage{enumerate}
\usepackage[toc,page]{appendix}
\usepackage{lineno}
\usepackage{color}
\usepackage{hyperref}
\usepackage[ruled,vlined,linesnumbered]{algorithm2e}

\SetKwInOut{Parameter}{\sl Algorithm Inputs}

\newcommand{\blue}{}
\definecolor{mygreen}{rgb}{0.1,0.75,0.2}

 \newtheorem{thm}{Theorem}[section]
 
 \newtheorem{lem}[thm]{Lemma}
 \newtheorem{prop}[thm]{Proposition}
 
 \newtheorem{rem}[thm]{Remark}

\numberwithin{equation}{section}

\DeclareMathOperator{\diam}{diam}

\newcommand{\Lra}{\Longrightarrow}

\newcommand{\la}{\langle}
\newcommand{\ra}{\rangle}
\newcommand{\pt}{\partial}
\newcommand{\eps}{\varepsilon}

\providecommand{\bbs}[1]{\left(#1\right)}
\newcommand{\aaa}[1]{\begin{equation}
begin{aligned} #1 \end{aligned}
\end{equation}}

\newcommand{\ud}{\,\mathrm{d}}
\newcommand{\8}{\infty}

\newcommand{\nn}{\mathcal{N}}

\newcommand{\my}{\mathbf{y}}

\newcommand{\hs}{\mathcal{H}}

\newcommand{\bR}{\mathbb{R}}

\newcommand{\sL}{\mathcal{L}}

\DeclareMathOperator*{\argmax}{argmax}
\DeclareMathOperator*{\argmin}{argmin}
\renewcommand{\vec}[1]{\ensuremath{\boldsymbol{#1}}}

\begin{document}

\title[Irreversible drift-diffusion  process]{Random walk approximation   for irreversible  drift-diffusion process on manifold:  ergodicity, unconditional stability and convergence}

\author[Y. Gao]{Yuan Gao}
\address{Department of Mathematics, Purdue University, West Lafayette, IN}
\email{gao662@purdue.edu}

\author[J.-G. Liu]{Jian-Guo Liu}
\address{Department of Mathematics and Department of
  Physics, Duke University, Durham, NC}
\email{jliu@math.duke.edu}


\begin{abstract}
Irreversible drift-diffusion processes are very common in biochemical reactions. They have a non-equilibrium stationary state (invariant measure)  which does not satisfy detailed balance. For the corresponding Fokker-Planck equation on a closed manifold, using Voronoi tessellation,  we propose two upwind finite volume  schemes with or without  the information of the invariant measure. Both schemes possess  stochastic $Q$-matrix structures and can be decomposed as a gradient flow part and a Hamiltonian flow part, enabling us to prove   unconditional stability,  ergodicity and error estimates. Based on the two upwind schemes, several numerical examples - including  sampling accelerated by a mixture flow, image transformations and simulations for stochastic model of chaotic system - are conducted. These two structure-preserving schemes also give a natural random walk approximation for a generic irreversible drift-diffusion process on a manifold. This makes them suitable for adapting to manifold-related computations that arise from high-dimensional molecular dynamics simulations.
\end{abstract}


\maketitle

\section{Introduction}

A general stationary (time-homogeneous) dynamical system with  white noise, can be modeled by  a stochastic differential equation for $\my_t\in \bR^\ell$
\begin{equation}\label{sde-x}
\ud \my_t = \vec{b}(\my_t) \ud t + \sqrt{2}\sigma  \ud B_t,
\end{equation}
where $\sigma$ is a   noise matrix and { $B_t$ is an $\ell$-dimensional Brownian motion}. 
Denote $D:=\sigma \sigma^T\in \bR^{\ell \times \ell}$. For simplicity, we assume $D$ is a constant positive semi-definite matrix. By Ito's formula,  SDE \eqref{sde-x} gives the following Fokker-Planck equation, which is the master equation for   the time marginal density $\rho_t(\my)$
\begin{equation}\label{FP-N}
\pt_t \rho =-\nabla\cdot(\vec{b} \rho)+ \nabla \cdot  ( D\nabla \rho )=: \sL^* \rho.
\end{equation}
In some physical systems,  the drift vector field $\vec{b}=-D\nabla \varphi$ for some  potential $\varphi$ representing the energy landscape. Then the Gibbs measure $\pi(\my)\propto e^{-\varphi(\my)}$ is the invariant measure.   The simplest example is the Ornstein-Uhlenbeck process with $\vec{b}(\my_t)=-\gamma \my_t$ and the diffusion coefficient $\sigma= \sqrt{\eps\gamma}$. In this case, \eqref{sde-x} is called  Langevin dynamics, and the corresponding Fokker-Planck equation has a gradient flow structure; see \eqref{KL}. In this Langevin dynamics case, the Markov process defined by \eqref{sde-x} is reversible\footnote{In some physics literature, it is referred as microscopic reversibility \cite{Onsager}.}, i.e., if we take $\pi$ as the initial density, the time-reversed process has the same law as that of  the forward process.  Equivalently, the invariant measure satisfies the detailed balance condition 
\begin{equation}
\text{steady flux }\, F^\pi := -\vec{b} \pi+ D \nabla\pi   = 0.
\end{equation} 

However,  numerous dynamical systems in physics and biochemistry are described by irreversible Markov processes (without detailed balance), i.e., there does \textit{not} exist a potential function such that the drift vector field $\vec{b}= -D\nabla \varphi$  in \eqref{sde-x}. For instance, the stochastic Lorenz system, the Belousov–Zhabotinsky reaction, or the Hodgkin–Huxley model  describe the  excitation and propagation  of sodium and potassium ions in a neuron. The irreversibility in the nonequilibrium circulation balance is almost literally the primary characteristic of life activities \cite{hill2005free}. In this case, the invariant measure $\pi$ is still stationary in time, but there is a positive  entropy production rate;  see \eqref{entropyP}. Thus \textsc{Prigogine}  named such an  invariant measure $\pi$ as   ``stationary non-equilibrium  states'' or ``non-equilibrium steady states'' in \cite[Chapter VI]{prigogine1968introduction}. We will simply call it steady state or invariant measure. Later, \text{Hill} explains Prigogine's theory  using Markov chain stochastic models for some simple biochemical reactions such as  muscle contraction and  clarifies the formula \eqref{entropyP} for the entropy production rate \cite[eq. (9.20)]{hill2005free}.  Another situation is that for  a Markov process on manifold, which is induced via dimension reductions (such as diffusion map \cite{CKLMN}) from a higher dimensional Markov process based on collected data, some classical schemes such as the Euler-Maruyama scheme will break the detailed balance property. 

Therefore, in this paper
we focus on designing numerical schemes  to simulate a general irreversible Markov process on a closed manifold $\nn$; see \eqref{sde-y}.   In terms of the SDE, we will design a random walk approximation which enjoys   ergodicity and accuracy. In terms of  Fokker-Planck equation \eqref{FP-N}, we will design two upwind schemes with a $Q$-matrix\footnote{a.k.a. infinitesimal generator matrix for a Markov chain} structure so that they also enjoy ergodicity, unconditionally stability and accuracy.

Assume $\nn$ is a $d$-dimensional closed manifold which is smooth enough. Let $\vec{b}\in T_\nn$ be a given tangent vector field. 
We denote the over-damped Langevin dynamics of $\my$ by
\begin{equation}\label{sde-y}
\ud {\my_t} = \vec{b}(\my_t) \ud t + \sqrt{2} \sigma \sum_{i=1}^d \tau^\nn_{i}(\my_t)\otimes \tau^\nn_i(\my_t) \circ \ud B_t,
\end{equation}
where  $\sigma\in \bR^{\ell \times \ell}$ is a constant matrix  corresponding to the thermal energy in physics, the symbol $\circ$ means the Stratonovich  integral, $B_t$ is $\ell$-dimensional Brownian motion and $\{\tau^\nn_i;\, 1\leq i \leq d\}$ are orthonormal basis of tangent plane $T_{\my_t}\nn$.
Here $\nabla_\nn := \sum_{i=1}^d \tau^{\nn}_i \nabla_{\tau^{\nn}_i}= \sum_{i=1}^d \tau^{\nn}_i \otimes \tau^{\nn}_i \nabla$ is the surface gradient and $\nabla_{\tau^{\nn}_i}=\tau^{\nn}_i \cdot \nabla$ is the tangential derivative in the direction of $\tau^{\nn}_i$.  More precise conditions on the manifold are described in \cite[Chapter 4]{ikeda}, \cite{hsu2002stochastic}. By Ito's formula, the corresponding Fokker-Planck equation is
\begin{equation*}
\pt_t \rho =-\nabla_\nn\cdot(\vec{b} \rho)+ \nabla_\nn \cdot  ( D\nabla_\nn \rho ).
\end{equation*}
For simplicity of notation, we drop subscript $\nn$ and still use \eqref{FP-N}.

In terms of  a general Fokker-Planck equation on a closed manifold $\nn$, if without detailed balance, we can  decompose it as a gradient flow part (described by a symmetric operator) and a Hamiltonian flow part (described by an antisymmetric operator); see Section \ref{sec_2.2_decom}. It is important to design numerical schemes that preserves this structure in the discrete sense. Thus in Section \ref{sec2}, based on the Voronoi tessellation for manifold $\nn$, we will develop two upwind finite volume schemes preserving stochastic $Q$-matrix structures and the discrete decompositions. 
The first upwind scheme does not rely on knowing the steady state $\pi$ of   irreversible dynamics \eqref{sde-y}; see Section \ref{sec_2.1_scheme1}. The second upwind scheme, which is called $\pi$-symmetric upwind scheme, leverages a given steady state information $\pi$ to simulate   irreversible dynamics \eqref{sde-y} and also enjoys   ergodicity to this given steady state; see Section \ref{sec_2.3_scheme2}. More importantly, schemes with  stochastic $Q$-matrix structures always have a $\pi$-symmetric decomposition, which decomposes the discrete flux as a symmetric part (corresponding to a dissipation part)
 and an antisymmetric part (corresponding to an energy-conservative part); see Section \ref{sec3}. The later part has no contribution to the discrete energy dissipation, so we have   same stability and ergodicity properties for both   schemes; see Proposition \ref{prop_energy} and Lemma \ref{lem_longtime}. Moreover, based on the $Q$-matrix structure, we will propose an unconditionally stable explicit time discretization and prove its stability and exponential ergodicity; see Proposition \ref{prop_timeD}.

In Section \ref{sec4}, we will give  convergence and  error estimates for the numerical steady state and the numerical dynamic solutions solved by the first upwind scheme \eqref{mp}, which rely on the Taylor expansion on the manifold and the discrete energy dissipation law; see Theorem \ref{thm_con}. In Section \ref{sec5}, several numerical examples based on   upwind scheme \eqref{mp} and $\pi$-symmetric upwind scheme \eqref{mp-pi} are conducted: (i) accelerated   sampling enhanced by an incompressible mixture flow; (ii) image transformations immersed in a mixture flow; and (iii) the stochastic Van der Pol oscillator in which  we only know  the  drift vector field $\vec{b}$ in the irreversible  process. In the sampling and image examples, it is interesting to see the convection (the Hamiltonian flow part in the scheme) brought by the  incompressible mixture flow speed up   convergence of the dynamic solution   to its steady state, which could be a promising direction to explore further in the future.

{ For  Fokker-Planck equation \eqref{FP-N} in a bounded open domain  $\Omega\subset\bR^\ell$ with various boundary conditions, there are many pioneering studies on  numerical  simulation. The most famous one is the Scharfetter-Gummel (SG) scheme proposed in \cite{scharfetter1969large} for some 1D semiconductor device equations; see Appendix \ref{app:SG} for detailed comparisons. Some extensions and mathematical analysis have been studied; e.g.,  \cite{markowich1985stationary, markowich1988inverse,  bank1998finite, xu1999monotone}
and recently in \cite{chainais2020large}. We will follow similar ideas of a finite volume method  for the drift-diffusion equation \cite{eymard2000finite, bessemoulin2012finite, chainais2003finite} but place more emphasis  on the new Markov chain structures and $\pi$-symmetric decomposition.  We develop the structure preserving schemes described above which possess stochastic $Q$-matrix structure and discrete $\pi$-symmetric decomposition. Particularly, our scheme also preserves   irreversibility and recovers the original irreversible invariant measure. In summary, for our scheme \eqref{mp-pi}, one has all the good properties including (i) positivity preserving; (ii) total mass preserving; (iii) well-balance property; (iv) $\ell^1$-contraction; (v) the discrete $\pi$-symmetric decomposition \eqref{decom};   (vi) energy dissipation law; and (vii) ergodicity.   Both of these schemes \eqref{mp} and \eqref{mp-pi} naturally provide a random walk approximation for irreversible Markov processes on manifolds and thus also more easily adaptable to some data-driven algorithms based on high dimensional point clouds; c.f., \cite{GLW20, GLLL21}. We also refer to  \cite{chainais2022long} for exponential ergodicity of a hybrid finite volume schemes and refer to \cite{schlichting2022scharfetter} for upwind schemes of aggregation--diffusion equations. The idea for the design of numerical schemes for these two recent work rooted in   \cite{schlichting2022scharfetter} but using a nonlinear approach which do  not have $Q$-matrix properties.
}

The remaining paper will be organized as follows. In Section \ref{sec2}, we propose two upwind schemes with $Q$-matrix structure for  the irreversible process \eqref{sde-y}. In Section \ref{sec3}, we give the discrete $\pi$-symmetric decomposition for a generic irreversible process on a closed manifold. Based on this, an energy dissipation law and exponential ergodicity are proved for both  schemes. In Section \ref{sec4}, we give  convergence and error estimates in terms of the $\chi^2$-divergence. In Section \ref{sec5}, several 
numerical examples with/without steady state information are presented. Two schemes in the 2D structured grids case with no-flux boundary condition are given in appendix for completeness.

\section{Two upwind schemes as random walk approximations for irreversible process}\label{sec2}
This section focuses on constructing random walk approximations for irreversible drift-diffusion process \eqref{sde-y}.  The approximations are proposed based on some upwind finite volume schemes for the corresponding Fokker-Planck equation \eqref{FP-N}. 

We focus on two kinds of fundamental problems in numerical simulations for irreversible process, i.e., the  drift vector field $\vec{b}$ does not satisfy the detailed balance condition $\vec{b}=-D\nabla \phi$.

The first kind of problem is we only know  the  drift vector field $\vec{b}$ without steady state information.
In this case, we will design an upwind scheme \eqref{mp} with $Q$-matrix structure based on the Voronoi tessellation for manifold $\nn$ to solve both  numerical steady state and simulate the dynamic process described by \eqref{FP-N}; see Section \ref{sec_2.1_scheme1}.

The second kind of problem is that in many applications, we have  information about the invariant measure $\pi$, although it is not detailed balanced (pointwise steady flux $F^\pi \neq 0$). This kinds of ``steady state with  nonzero flux on network'' happens very common in biochemistry. In this case, we will design a $\pi$-symmetric upwind   scheme \eqref{mp-pi} which recovers the given steady state $\pi$; see Section \ref{sec_2.3_scheme2}. This scheme is motivated by a reformulation as a gradient flow part and a Hamiltonian flow part for the continuous Fokker-Planck equation; see Section \ref{sec_2.2_decom}. Detailed analysis  for upwind schemes with generic $Q$-matrix structure will be given in Section \ref{sec3}.

\subsection{Voronoi tessellation and upwind finite volume scheme}\label{sec_2.1_scheme1}
In this section, we first propose a finite volume scheme for the  Fokker-Planck equation \eqref{FP-N} based on a   Voronoi tessellation  for $\nn$. Then we design an upwind finite volume scheme with a $Q$-matrix structure, which can be	 reformulated  as a Markov  process on finite sites and enjoys good properties. 

Suppose $(\nn, d_\nn)$ is a $d$ dimensional smooth closed submanifold of $\mathbb{R}^\ell$ and $d_{\nn}$ is  induced by the Euclidean metric in $\bR^\ell$. $S:=\{\my_i\}_{i=1:n} $ are point clouds sampled from a density function on $\nn$ bounded below and above. It is proved that the data points $S$ are well-distributed on $\nn$  whenever the points are sampled from a density function with lower and upper bounds \cite{Dejan15}. Define the Voronoi cell as
\begin{equation}
C_i:= \{\my\in \nn ; \ud_\nn(\my,\my_i)\leq \ud_\nn(\my,\my_j) \text{ for all }\my_j\in S\} \quad  \text{ with volume } |C_i|=\hs^d(C_i).
\end{equation}
{ Here $\hs^d(C_i)$ the $d$-dimensional Hausdorff measure of cell $C_i$.}
Then $\nn=\cup_{i=1}^n  C_i$ is a Voronoi tessellation of  $\nn$. Denote the Voronoi face for cell $C_i$ as
\begin{equation}
\Gamma_{ij}:= C_i\cap C_j   \text{ with its area  } |\Gamma_{ij}|=\hs^{d-1}(\Gamma_{ij}),
\end{equation}
for any $j=1, \cdots, n$. If $\Gamma_{ij}= \emptyset$ or $i =  j$ then we set $|\Gamma_{ij}|=0$. 
Define the associated adjacent sample points as
\begin{equation}
VF(i):=\{j; ~\Gamma_{ij}\neq \emptyset\}.
\end{equation}

Using the above Voronoi tessellation, associated the discrete density $\rho_i|C_i|$ at cell $C_i$ and the sign of the flux at each site, we design the upwind scheme as follows.

For each cell $C_i$, denote the unit outer normal vector field on $\pt C_i$ (pointing from $i$ to its adjacent $j$) as $\vec{n}\in \bR^\ell$.     Denote
$
(\vec{b} \cdot \vec{n})^+_{ij}, \, (\vec{b} \cdot \vec{n})^-_{ij}>0
$
as the positive and negative parts of $(\vec{b} \cdot \vec{n})_{ij}$ respectively, where $(\vec{b} \cdot \vec{n})_{ij}$ means evaluate $(\vec{b} \cdot \vec{n})$ at the intersection point of the  geodesic from $\my_i$ to $\my_j$. 
We integrate \eqref{FP-N} on $C_i$ and use the divergence theorem on cell $C_i$ to obtain
\begin{equation}\label{tt310}
\frac{\ud }{\ud t} \int_{C_i} \rho \hs^d(C_i) = \sum_{j\in VF(i) } \int_{\Gamma_{ij}}  \mathbf n \cdot \bbs{-\vec{b} \rho + D \nabla \rho} \hs^{d-1}(\Gamma_{ij}).
\end{equation}
The 
probability in $C_i$ can be approximated as
\begin{equation}\label{alg11}
\int_{C_i} \rho(\my)  \hs^d(C_i) \approx \rho(\my_i) (1+\diam (C_i)) |C_i|,
\end{equation}
so we use $\rho_i$ to approximate the exact solution $\rho(\my_i)$ on each cell $C_i$.

We introduce the following upwind finite volume scheme and call it ``upwind scheme''.
For $i=1, \cdots, n$,
\begin{equation}\label{mp}
\frac{\ud}{\ud t}\rho_i |C_i|=  \sum_{j\in VF(i)} |\Gamma_{ij}| \bbs{\frac{\vec{n}_{ij}\cdot D\vec{n}_{ij} \bbs{\rho_j-\rho_i}}{|\my_j-\my_i|}   + (\vec{b}\cdot \vec{n})_{ij}^- \rho_j -  (\vec{b}\cdot \vec{n})_{ij}^+ \rho_i },
\end{equation}
where $|C_i|$ is the volume element at cell $C_i$.

One can recast \eqref{mp} as a matrix form
\begin{equation}\label{rhoDeq}
\frac{\ud}{\ud t}\rho_i { |C_i|} = \sum_j Q^*_{ij} \rho_j|C_j|, \quad i=1, \cdots,n,
\end{equation}
where the $Q^*$-matrix is given by
\begin{equation}
Q^*_{ij} =
 \frac{|\Gamma_{ij}|}{|C_j|} \bbs{\frac{\vec{n}_{ij}\cdot D\vec{n}_{ij}}{|\my_j-\my_i|}  + \bbs{\vec{b} \cdot n}^-_{ij} }\geq 0, \quad j\neq i, 
\qquad Q^*_{ii}= \sum_{j\in VF(i) } \frac{|\Gamma_{ij}|}{|C_i|} \bbs{-\frac{\vec{n}_{ij}\cdot D\vec{n}_{ij}  }{|\my_j-\my_i|}  - \bbs{\vec{b} \cdot \vec{n}}^+_{ij}  }. 
\end{equation}
Since for any two adjacent $i$ and $j$, we have 
\begin{equation}\label{sys}
\vec{n}_{ij} = - \vec{n}_{ji}, \quad  \bbs{\vec{b} \cdot \vec{n}}^-_{ji} = \bbs{\vec{b} \cdot \vec{n}}^+_{ij} .
\end{equation}
Thus the transport of $Q^*$ satisfies
\begin{equation}\label{Qm}
\begin{aligned}
& Q_{ij} =
 \frac{|\Gamma_{ij}|}{|C_i|} \bbs{\frac{ \vec{n}_{ij}\cdot D\vec{n}_{ij}  }{|\my_j-\my_i|} + \bbs{\vec{b} \cdot \vec{n}}^-_{ji} } =  \frac{|\Gamma_{ij}|}{|C_i|} \bbs{\frac{ \vec{n}_{ij}\cdot D\vec{n}_{ij}  }{|\my_j-\my_i|} + \bbs{\vec{b} \cdot \vec{n}}^+_{ij} } \geq 0 , \quad j\neq i, \\
& Q_{ii}= -\sum_{j\in VF(i) }Q_{ij} .
\end{aligned}
\end{equation}
 One can see $Q$-matrix is a stochastic matrix that row sums zero; see \cite[Definition 2.3]{liggett2010}.
  Then $Q$ is the generator of the associated Markov process on point clouds.
We list the following standard properties for $Q$-process \eqref{rhoDeq}, which  guarantee good properties for the numerical scheme:
\begin{itemize}\label{sss}
\item positivity preserving, i.e., $\min_i\rho_i(0)\geq 0 \Lra \min_i \rho_i(t)\geq 0$;
\item total mass preserving, i.e., $\frac{\ud}{\ud t} \sum_i \rho_i|C_i| = \sum_{i,j} Q^*_{ij}\rho_j =0$;
\item $\ell^1$-contraction, i.e., $\frac{\ud}{\ud t} \sum_i |\rho_i - \tilde{\rho}_i||C_i| \leq  \sum_{i,j} Q^*_{ij}|\rho_j - \tilde{\rho}_j| = 0   $.
\end{itemize}
For the adjoint process $\pt_t f_i = \sum_j Q_{ij}f_j$, it satisfies maximal principle, i.e., for $i_m:=\argmin_i f_i$ and $i_M:=\argmax_i f_i,$ 
$$\frac{\ud}{\ud t}  f_{i_m} = \sum_j Q_{i_m j}\bbs{f_j-f_{i_m}}  \geq 0 ; \qquad \frac{\ud}{\ud t}  f_{i_M} = \sum_j Q_{i_M j}\bbs{f_j-f_{i_M}}  \leq 0.$$
When the exact metric for the manifold is unknown, we can also use collected point clouds which probe the manifold to compute an approximated Vonoroi tessellation; see details in \cite[Algorithm 1]{GLW20}. { We also give a comparison between our scheme and previous finite volume schemes in Appendix \ref{app:SG}.}
\begin{rem}
One can interpret the upwind finite volume scheme \eqref{mp} as
 the forward equation for a Markov process with
 transition probability $P_{ji}$ (from $j$ to $i$) and jump rate $\lambda_j$
\begin{equation}\label{mp1}
\frac{\ud}{\ud t}\rho_i |C_i| = \sum_{j\in VF(i)} \lambda_j P_{ji} \rho_j |C_j| - \lambda_i \rho_i |C_i|,\quad i=1, 2, \cdots, n,
\end{equation}
where for  $i=1, 2, \cdots, n,$
\begin{equation}\label{def59}
\begin{aligned}
\lambda_i := \sum_{j\neq i} Q_{ij}, \quad 
 P_{ij}:=\frac{Q_{ij}}{\lambda_i}, \quad j\in VF(i); \quad P_{ij}=0, \quad j\notin VF(i).
 \end{aligned}
\end{equation}
 Using the exponential distribution $f(t;\lambda)=\lambda e^{-\lambda t}$ with rate $\lambda$, {\blue one common construction of $Q$-process with generator $Q$ is given by Gillespie's algorithm in the Monte Carlo simulation, i.e.,  given $X_t=i$, the probability for the event $\{ \text{after waiting time } \tau_i, \text{ the jump } i \text{ to } j, X_{t+\tau}=j  \text{ happens} \} $ is given by 
$
 P_{ij} \, \lambda_i e^{-\lambda_i \tau}.
$
This time-continuous Markov chain is an exact construction for the $Q$-process. There are also other time-discrete Markov chain constructions, which serve as good approximations for $Q$-process when time step $\Delta t\to 0$. For instance, let us assume (i) $X_t=i$, (ii) the probability for $X_t$ stays at site $i$ is $e^{-\lambda_i}\Delta t \approx 1-\lambda_i \Delta t$,  (iii) the probability for the jump $i$ to $j$ happens in $[t,t+\Delta t]$ is $\lambda_i \Delta t e^{-\lambda_i \Delta t} \approx\lambda_i \Delta t$. Then the corresponding master equation is just   the forward Euler scheme for \eqref{mp1}
\begin{equation}
\rho^{k+1}_{i} = \rho_i^k (1-\lambda_i \Delta t) + \sum_j \rho_j^k P_{ji}\lambda_j \Delta t.
\end{equation}
The unconditionally stable explicit scheme \eqref{timeG} is another example for the time-discrete Markov chain construction of the $Q$-process.
}
\end{rem}

\begin{rem}
The upwind finite volume scheme \eqref{mp} belongs to monotone schemes uniform in $D$. It is well known 	all monotone schemes that uniform in $D$ have at most first order accuracy. To achieve higher order schemes for convection-dominated problems that enjoy the above positivity preserving property and still uniform in $D$, one need to restore  some nonlinear schemes by  the method of limiter or streamline diffusion methods.  However, the nonlinearity will destroy the $Q$-matrix structure and we will leave  high order schemes for a future study. On the other hand,  in the case that diffusion $D$ is not small, standard centered scheme can be adapted as a second order scheme. 
\end{rem}

\subsection{Reformulate as gradient flow structure and Hamiltonian structure }\label{sec_2.2_decom}
Recall the continuous Fokker-Planck operator $\sL^*$ in \eqref{FP-N}.
From now on, to ensure existence of a positive invariant measure $\pi$,  we assume $D$ is positive definite. Based on \cite[Proposition 4.5]{ikeda}, we know  the Fokker-Planck operator $\sL^*$ on the compact manifold $\nn$ has a unique invariant measure, denoted as $\pi$, 
\begin{equation}\label{FPequi}
\begin{aligned}
0=\sL^* \pi = -\nabla \cdot(\vec{b} \pi) + \nabla \bbs{D  \nabla \pi}.
\end{aligned}
\end{equation}

We first use the invariant measure $\pi$ to decompose \eqref{FP-N} as two parts: gradient flow part and Hamiltonian flow part.
Rewrite \eqref{FP-N} as
\begin{equation}\label{FPr}
\begin{aligned}
\pt_t \rho =&\sL^* \rho =\nabla \cdot \bbs{D \nabla \rho - \vec{b} \rho}
= \nabla \cdot \bbs{ D \pi \nabla \frac{\rho}{\pi} +  \frac{\rho}{\pi} \bbs{D\nabla \pi -   \pi \vec{b}} }\\
 =& \nabla \cdot \bbs{ D \pi \nabla \frac{\rho}{\pi}} +   \bbs{D\nabla \pi -   \pi \vec{b}} \cdot \nabla \frac{\rho}{\pi} 
 =: L^* \frac{\rho}{\pi} + T \frac{\rho}{\pi}.
\end{aligned}
\end{equation} 
Here the symmetric operator $L^*$ is 
\begin{equation}
L^*:= \nabla \cdot\bbs{D \pi \nabla}, \quad \text{with }\,  \la f, L^* g\ra = \la L^* f, g\ra,  \quad \la u, L^* u \ra \leq 0,
\end{equation}
while the antisymmetric operator $T$ is 
\begin{equation}
  T:= \bbs{D\nabla \pi -   \pi \vec{b}} \cdot \nabla=: \vec{u}\cdot \nabla, \quad \text{ with }\,  \la f, T g \ra = -\la T f, g\ra
\end{equation}
since $\nabla \cdot \bbs{D\nabla \pi -   \pi \vec{b}} = \nabla \cdot \vec{u}=0$ by \eqref{FPequi}. Here we remark that in terms of $\rho$-variable, we usually say $\nabla \cdot\bbs{D \pi \nabla\frac{1}{\pi}}$ (resp. $\vec{u} \cdot \nabla \frac{1}{\pi}$) is symmetric (resp. antisymmetric) operator in $L^2(\frac{1}{\pi}).$ We will call this decomposition as `$\pi$-symmetric decomposition'.

In the irreversible case,   \eqref{FPr} for the reformulated Fokker-Planck equation can be regarded as a gradient flow part $\pt_t \rho = L^* \frac{\rho}{\pi} =\nabla \cdot \bbs{ D \pi \nabla \frac{\rho}{\pi}} $ plus a Hamiltonian flow part (convection part)
$\pt_t \rho = T \frac{\rho}{\pi} = \bbs{D\nabla \pi -   \pi \vec{b}} \cdot \nabla \frac{\rho}{\pi}$. Indeed, 
given a convex function $\phi$ with $\phi''\geq 0$, denote  the free energy as
\begin{equation}\label{EE}
E := \int \phi\bbs{\frac{\rho}{\pi}} \pi \ud x.
\end{equation} 
Since $\nabla \cdot \bbs{D\nabla \pi -   \pi \vec{b}}=0$, for any free energy of the form \eqref{EE}, we have
\begin{equation}\label{ham}
\la \frac{\delta E}{\delta \rho}, T\frac{\rho}{\pi} \ra = \la \phi'\bbs{\frac{\rho}{\pi}}, T \frac{\rho}{\pi} \ra =\int \bbs{D\nabla \pi -   \pi \vec{b}} \cdot \nabla \phi\bbs{\frac{\rho}{\pi}} \ud x = 0.
\end{equation}
For the gradient flow part, we have
\begin{equation}
\la \frac{\delta E}{\delta \rho}, L^*\frac{\rho}{\pi} \ra = \la \phi'\bbs{\frac{\rho}{\pi}}, L^* \frac{\rho}{\pi} \ra =-\int  \phi''\bbs{\frac{\rho}{\pi}} \pi \nabla \frac{\rho}{\pi}\cdot D \nabla \frac{\rho}{\pi}   \ud x \leq 0.
\end{equation}
  This observation is quite similar to the so called  GENERIC (general equation for non-equilibrium reversible-irreversible coupling) formalism, which also decomposes a general thermodynamics as a gradient flow part and a Hamiltonian flow part.

From \eqref{ham}, in terms of the energy dissipation law, the Hamiltonian part has no contribution. 
From now on, we denote $c$ as a generic constant whose value may change from line to line. 

Below we summarize the following lemma for the energy dissipation law for continuous equation and later we will use it to derive a corresponding discrete energy dissipation law.
\begin{lem}\label{lem_E}
Let $E$ be an entropy   defined  in \eqref{EE}, we have the energy dissipation law
\begin{equation}
\frac{\ud E}{\ud t} =  -\int \phi''\bbs{\frac{\rho}{\pi}}\pi \nabla \frac{\rho}{\pi}\cdot D \nabla \frac{\rho}{\pi} \ud x\leq 0.
\end{equation}
\end{lem}

Specially, it is well known that the reversible condition (detailed balance condition) is equivalent to $\vec{b}=D \nabla \log \pi$, i.e., there is only symmetric part $\pt_t \rho = L^* \frac{\rho}{\pi}$. 
In this case, there are particularly two well known gradient flow structures:
\begin{enumerate}[(i)]
\item Take $\phi(x)=x\log x$ then free energy becomes $\text{KL}(\rho||\pi) = \int \rho \log \frac{\rho}{\pi} \ud x$, and we have the gradient flow
\begin{equation}\label{KL}
\pt_t \rho = \nabla \cdot \bbs{\rho D \nabla \frac{\delta \text{KL}}{\delta \rho}}; \quad \frac{\ud}{\ud t} \text{KL}(\rho||\pi)  = - \la \rho   \nabla \log \frac{\rho}{\pi}, D \nabla \log \frac{\rho}{\pi} \ra \leq 0;
\end{equation}
\item Take $\phi(x) =\frac12 x^2$, then free energy becomes  $\chi^2$-divergence $\chi^2(\rho)=\frac12 \int \frac{\rho^2}{\pi} \ud x$, and we have the gradient flow
\begin{equation}
\pt_t \rho = \nabla \cdot \bbs{\pi D \nabla \frac{\delta \chi^2(\rho)}{\delta \rho}}; \quad \frac{\ud}{\ud t} \chi^2(\rho)  = - \la \pi \nabla \frac{\rho}{\pi},  D \nabla \frac{\rho}{\pi} \ra \leq 0;
\end{equation}
Particularly, take $\phi(x) =\frac12 (x-1)^2$, then we have the decay estimate for $\frac12 \int \frac{(\rho-\pi)^2}{\pi} \ud x$, 
\begin{equation}
 \frac{\ud}{\ud t} \frac12 \int \frac{(\rho-\pi)^2}{\pi} \ud x  = -  \la \pi \nabla \frac{\rho}{\pi},  D \nabla \frac{\rho}{\pi} \ra\leq 0.
\end{equation}
This together with Poincare's inequality 
$$\int |u|^2 \pi \ud x \leq c \la \pi \nabla u, D \nabla u \ra \quad \text{ for } \int u \pi \ud x =0,$$ 
yields the exponential decay of $\rho$ to steady state $\pi$. 
\end{enumerate}

Thanks to   Poincare's inequality,   $\pi$-symmetric decomposition \eqref{FPr} and Lemma \ref{lem_E}, we conclude that in terms of exponential ergodicity, there is no difference between the irreversible process and the corresponding  reversible process.
Indeed,  the equivalence of the exponential ergodicity for the irreversible process and the corresponding  reversible process was already established in \cite{Chen_2000}; particularly for countable state space. Moreover, besides the gradient flow part, the incompressible transport $\vec{u}\cdot\nabla \frac{\rho}{\pi}$ brought by $\vec{u}=D\nabla \pi -\pi \vec{b}$ usually results in mixture. Thus we will observe a speedup of convergence for the irreversible process; see Section \ref{sec_5.1pi}.

From  $\pi$-symmetric decomposition \eqref{FPr}, for the irreversible case, the dynamics can not be described only by a gradient flow and there is an additional Hamiltonian flow part. This observation is also given by \cite[Example 4.3]{Mielke_Renger_Peletier_2014}. Moreover, by defining a $L$-function in the large deviation principle, \cite{feng2006large, Mielke_Renger_Peletier_2014} established the relation between (generalized) gradient flow and the $L$-function in the large deviation principle; see also \cite{yg20, GLL21, gao2022thermodynamic}.

Designing a structure preserving numerical scheme is important and will be studied for schemes with generic $Q$-matrix structure  in   Section \ref{sec3} via a discrete $\pi$-symmetric decomposition.
Specifically, we will show both  ``upwind scheme'' \eqref{mp} and ``$\pi$-symmetric upwind scheme'' \eqref{mp-pi} are  structure preserving numerical scheme which enjoy the above $\pi$-symmetric decomposition,  a discrete dissipation law and  exponential ergodicity.

\subsection{$\pi$-symmetric upwind scheme for irreversible processes with an invariant measure}\label{sec_2.3_scheme2}

In this section, we design a finite volume scheme based on  $\pi$-symmetric decomposition \eqref{FPr} and based on a given steady state  information, i.e., the invariant measure $\pi$.  The expected numerical scheme, which defines a numerical $\pi_i^\8$, should recover the given steady state $\pi_i^\8=\pi_i.$ This kind of idea that preserves given steady state is known as ``well-balanced scheme.''
Recall the symmetric decomposition \eqref{FPr}
\begin{equation}
\begin{aligned}
\pt_t \rho =\nabla \cdot \bbs{D \nabla \rho - \vec{b} \rho}
 = \nabla \cdot \bbs{ D \pi \nabla \frac{\rho}{\pi}} +   \bbs{D\nabla \pi -   \pi \vec{b}} \cdot \nabla \frac{\rho}{\pi} 
 =: L^* \frac{\rho}{\pi} + T \frac{\rho}{\pi}.
\end{aligned}
\end{equation}
From stationary equation \eqref{FPequi}, the new drift velocity
\begin{equation}\label{nb}
\vec{u}:= D\nabla \pi - \pi \vec{b}, \quad \nabla \cdot \vec{u} =0.
\end{equation}

Then two common scenarios in applications are  (i) given the non-gradient form drift $\vec{b}$ and steady state $\pi$, we can compute the new drift velocity \eqref{nb}; (ii) given any incompressible velocity field $\vec{u}$ such that $\nabla \cdot \vec{u}=0$, based on $\pi$, then the original drift is given by
$
\vec{b}  = D \frac{\nabla \pi}{\pi} + \frac{\vec{u}}{\pi}.
$

Now using the symmetric decomposition \eqref{FPr} and   the given steady state  information $\pi$, we design the following upwind finite volume scheme and call it ``$\pi$-symmetric upwind scheme''.
For $i=1, \cdots, n$,
\begin{equation}\label{mp-pi}
\frac{\ud}{\ud t}\rho_i |C_i|=  \sum_{j\in VF(i)} |\Gamma_{ij}| \bbs{\frac{D(\pi_i+\pi_j) }{2|\my_j-\my_i|}\bbs{\frac{\rho_j}{\pi_j}-\frac{\rho_i}{\pi_i}}   + (\vec{u}\cdot \vec{n})_{ij}^- \frac{\rho_j}{\pi_j} -  (\vec{u}\cdot \vec{n})_{ij}^+ \frac{\rho_i}{\pi_i} },
\end{equation}
where $|C_i|$ is the volume element at cell $C_i$. 
Here $
(\vec{u} \cdot \vec{n})^+_{ij}, \, (\vec{u} \cdot \vec{n})^-_{ij}>0
$
as the positive and negative parts of $(\vec{u} \cdot \vec{n})_{ij}$ respectively.

Denote
\begin{equation}\label{Q_tn}
\begin{aligned}
Q^*_{ij}= \frac{|\Gamma_{ij}|}{|\pi_j||C_j|} \bbs{\frac{ D(\pi_i+\pi_j)  }{2|\my_j-\my_i|} + \bbs{\vec{u} \cdot \vec{n}}^-_{ij} }  \geq 0,\,\, j\neq i,\\ \quad  Q^*_{ii}= \sum_j \frac{|\Gamma_{ij}|}{|\pi_i||C_i|} \bbs{-\frac{ D(\pi_i+\pi_j)  }{2|\my_j-\my_i|} - \bbs{\vec{u} \cdot \vec{n}}^+_{ij} } 
\end{aligned}
\end{equation}
Then from \eqref{sys}
\begin{equation}\label{Q-p}
Q^*_{ij}=Q_{ji}=  \frac{|\Gamma_{ij}|}{|\pi_j||C_j|} \bbs{\frac{ D(\pi_i+\pi_j)  }{2|\my_j-\my_i|} + \bbs{\vec{u} \cdot \vec{n}}^-_{ij} }  \geq 0,\,\, j\neq i, \quad  Q_{ii} = - \sum_{j\neq i} Q_{ij},
\end{equation}
which is still a stochastic  $Q$-matrix that row sums zero.
Then \eqref{mp-pi} can be recast as
\begin{equation}
\frac{\ud}{\ud t}\rho_i |C_i| = \sum_j Q_{ij}^* \rho_j |C_j|. 
\end{equation}

\subsubsection{Structure preserving decomposition}
For the upwind part of \eqref{mp-pi}, notice
\begin{equation}
\begin{aligned}
&(\vec{u}\cdot \vec{n})^+_{ij}  \frac{\rho_j}{\pi_j} 
= (\vec{u}\cdot \vec{n})_{ij}   \frac{\rho_i}{\pi_i} +  (\vec{u}\cdot \vec{n})^-_{ij}   \frac{\rho_i}{\pi_i}.
\end{aligned}
\end{equation}
Then the flux $F_{ji}$ in \eqref{mp-pi} can be recast as
\begin{equation}\label{F_dec0}
F_{ji} =   |\Gamma_{ij}| \bbs{ \bbs{\frac{D(\pi_i+\pi_j) }{2|\my_j-\my_i|} + (\vec{u}\cdot \vec{n})_{ij}^-}\bbs{\frac{\rho_j}{\pi_j}-\frac{\rho_i}{\pi_i}}   -(\vec{u}\cdot \vec{n})_{ij}   \frac{\rho_i}{\pi_i}}.
\end{equation} 
Notice another anti-symmetric decomposition for the upwind parts
\begin{equation}
\begin{aligned}
&(\vec{u}\cdot \vec{n})_{ij}^- \frac{\rho_j}{\pi_j} -  (\vec{u}\cdot \vec{n})_{ij}^+ \frac{\rho_i}{\pi_i} 
= -\frac12  (\vec{u}\cdot \vec{n})_{ij} \bbs{\frac{\rho_j}{\pi_j} + \frac{\rho_i}{\pi_i}} + \frac12 | (\vec{u}\cdot \vec{n})_{ij}| \bbs{\frac{\rho_j}{\pi_j} - \frac{\rho_i}{\pi_i}}.
\end{aligned}
\end{equation}
Then another recast of flux $F_{ji}$ in \eqref{mp-pi} is
\begin{equation}\label{F_dec}
F_{ji} =   |\Gamma_{ij}| \bbs{ \bbs{\frac{D(\pi_i+\pi_j) }{2|\my_j-\my_i|} + \frac12 | (\vec{u}\cdot \vec{n})_{ij}|}\bbs{\frac{\rho_j}{\pi_j}-\frac{\rho_i}{\pi_i}}   -\frac12(\vec{u}\cdot \vec{n})_{ij}    \bbs{\frac{\rho_j}{\pi_j} + \frac{\rho_i}{\pi_i}}}.
\end{equation} 

Using \eqref{F_dec}, multiplying \eqref{mp-pi} by $\frac{\rho_i}{\pi_i}$ and taking summation w.r.t $i$, we have
\begin{equation}\label{L2decay1}
\begin{aligned}
 \frac{\ud}{\ud t} \frac{\rho_i^2}{\pi_i} |C_i| = -\sum_{i,j}  \frac{|\Gamma_{ij}|}{2}  \bbs{\frac{D(\pi_i+\pi_j) }{|\my_j-\my_i|} + | (\vec{u}\cdot \vec{n})_{ij}| }\bbs{\frac{\rho_j}{\pi_j}-\frac{\rho_i}{\pi_i}}^2 + I_e,\\
 I_e:= - \sum_{i,j} \frac{|\Gamma_{ij}|}{2} (\vec{u}\cdot \vec{n})_{ij} \bbs{\frac{\rho_j}{\pi_j} \frac{\rho_i}{\pi_i} + \frac{\rho_i^2}{\pi_i^2}}=- \sum_{i,j}\frac{|\Gamma_{ij}|}{2} (\vec{u}\cdot \vec{n})_{ij} \frac{\rho_i^2}{\pi_i^2}.
 \end{aligned}
\end{equation}
Notice in the continuous case, given $\pi$ and $\vec{b}$, the drift $\vec{u}$ in \eqref{nb} is divergence free. If we have same  divergence free condition in the discrete case,  then $I_e=0$ and we obtain the accurate convergence to steady state.  Thus we construct the discrete velocity field such that the incompressible condition  holds
\begin{equation}\label{div0}
\sum_j |\Gamma_{i,j}| (\vec{u} \cdot \vec{n})_{i,j} =0.
\end{equation}

There are some well-known schemes in computations for incompressible flows, Maxwell's equations and   magnetohydrodynamics \cite{weinan1996vorticity, liu2000high}. We adapt them in the 2D case as follows.
In 2D, there exists a stream function $\psi(x,y)$ such that $\vec{u}=(-\pt_y \psi, \pt_x \psi)$. For each 2D cell $C_i$, we set a counterclockwise orientation for its face $\Gamma_{ij}, \, j\in VF(i)$. Then at each face $\Gamma_{ij}$, denote the starting point as $\alpha_{ij}$ and the ending point as $\beta_{ij}$. Then the discrete incompressible velocity field is defined as 
\begin{equation}\label{streamU}
(\vec{u} \cdot \vec{n})_{ij}: = \pt_\tau \psi \approx \frac{\psi(\beta_{ij})-\psi(\alpha_{ij})}{|\beta_{ij}-\alpha_{ij}|} = \frac{\psi(\beta_{ij})-\psi(\alpha_{ij})}{|\Gamma_{ij}|},
\end{equation}
which automatically satisfies \eqref{div0}.
In 3D case, we need use a vector potential instead of stream function while in higher dimension there are more freedoms to construct the incompressible velocity.

Thanks to  \eqref{div0} and $I_e=0$, \eqref{L2decay1} becomes the discrete energy dissipation law
\begin{equation}
\frac{\ud}{\ud t} \frac{\rho_i^2}{\pi_i} |C_i| =-\frac12 \sum_{i,j} \alpha_{i,j} \bbs{\frac{\rho_j}{\pi_j}-\frac{\rho_i}{\pi_i}}^2 \leq 0
\end{equation}
with a symmetric coefficient $\alpha_{ij}= |\Gamma_{ij}| \bbs{\frac{D(\pi_i+\pi_j) }{|\my_j-\my_i|} + | (\vec{u}\cdot \vec{n})_{ij}| }$; see Proposition \ref{prop_energy} for same result with a general $Q$-matrix. Then  with the truncation error estimate, we will give the convergence analysis  in Theorem \ref{thm_con} for the numerical solution in terms of the $\chi^2$-divergence.

After choosing $\vec{u}$ satisfying \eqref{div0}, from \eqref{F_dec0},   $\pi$-symmetric upwind scheme \eqref{mp-pi} is equivalent to
\begin{equation}\label{F_num}
\frac{\ud}{\ud t} \rho_i |C_i| = \sum_{j\in VF(i)} F_{ji}, \quad  F_{ji} =   |\Gamma_{ij}|   \bbs{\frac{D(\pi_i+\pi_j) }{2|\my_j-\my_i|} + (\vec{u}\cdot \vec{n})_{ij}^-}\bbs{\frac{\rho_j}{\pi_j}-\frac{\rho_i}{\pi_i}}. 
\end{equation}
Thus although the original Fokker-Planck \eqref{FP-N} is irreversible, but the discrete  flux $F_{ji}$ can be chosen such that the discrete steady  flux
\begin{equation}\label{dis_db}
F^\pi_{ji} =  |\Gamma_{ij}|   \bbs{\frac{D(\pi_i+\pi_j) }{2|\my_j-\my_i|} + (\vec{u}\cdot \vec{n})_{ij}^-}\bbs{\frac{\pi_j}{\pi_j}-\frac{\pi_i}{\pi_i}} \equiv 0, \quad \forall i,j.
\end{equation}
This implies the scheme is well balanced, i.e., numerical steady state $\pi_i^\8=\pi_i.$
\begin{rem}
Although we can choose the discrete flux satisfies \eqref{dis_db},   we remark this does not contradict with the equivalent irreversible condition, i.e., the flux $F^\pi=0$ pointwisely   for continuous Fokker-Planck equation; c.f. \cite{qian2002thermodynamics, GLL21}. Indeed, because of different decomposition of flux and \eqref{div0}, in the numerical schemes, one can either use the pointwise flux $F_{ji}$ in \eqref{F_num} or use original $F_{ji}$ in \eqref{F_dec0}; however, only the former one satisfies \eqref{dis_db}.
\end{rem}

{ \begin{rem}
Although we design numerical schemes under the assumption $\nn$ is a closed manifold, for manifold with boundaries, we point out the construction of schemes and the analysis for no-flux boundary condition are identically same because $j\in VF(i)$ in both schemes \eqref{mp} and \eqref{mp-pi} already accommodate the no-flux boundary conditions for $i$ adjacent to the boundary. In the numerical examples in Section \ref{sec5}, for 2D structured grids,   schemes \eqref{mp} and \eqref{mp-pi} naturally equip with no-flux boundary conditions.
\end{rem}
}

\section{Structure preserving,   ergodicity and stability}\label{sec3}
In this section, for numerical schemes with an abstract generic $Q$-matrix structure, we present a structure preserving reformulation, which is the discrete counterpart of  $\pi$-symmetric decomposition \eqref{FPr}. Then we use it to prove   stability and ergodicity for schemes with a generic $Q$-matrix structure. These results apply to both \eqref{mp} and \eqref{mp-pi}.
\subsection{Structure preserving for the $\pi$-symmetric decomposition in discrete case} In this section, we leverage the numerical steady state $\pi^\8_i$ to recover the $\pi$-symmetric decomposition for  numerical schemes with a generic $Q$-matrix structure, for instance \eqref{mp}, and derive the discrete energy dissipation law. { We will show that all the structures for continuous equation are also preserved for our numerical scheme including (i) positivity preserving; (ii) total mass preserving; (iii) well-balance property; (iv) $\ell^1$-contraction; (v) the discrete $\pi$-symmetric decomposition \eqref{decom};  and (vi) energy dissipation law. } 

{ Step 1. As we already observed in \eqref{Q-p}, $Q$ defined in \eqref{Q-p} is a stochastic $Q$-matrix. Thus (i) positivity preserving and  (ii) the total mass preserving are straightforward. The well-balanced property directly comes from \eqref{dis_db}. For (iv), one can see for any $\rho_i$ and $\tilde{\rho}_i$ and $e_i= \rho_i-\tilde{\rho}_i$
\begin{equation}
\frac{\ud }{\ud t} \sum_i |e_i| = \sum_{i,j} Q_{ji} \text{sgn}(e_i) e_j \leq \sum_{i,j} Q_{ji} |e_j| =0, 
\end{equation}
due to $\sum_j Q_{ij}=0$. Thus (iv) $\ell^1$-contraction holds.
}

Step 2. Before  doing the symmetric decomposition, we first clarify the existence and uniqueness of a positive numerical steady state $\pi^\8$.  { Since our scheme is well-balanced \eqref{dis_db}, the numerical steady state $\pi^\8=\pi$ exists. Next, we show the uniqueness of $\pi^\8$ using the Perron-Frobenius
theorem.}  Let $Q$ be the $Q$-matrix defined in \eqref{Qm}. Notice the Vonoroi tessellation implies the directed graph associated to $Q$ is connected.  Thus $Q$ is irreducible. Taking a constant $a>\max_i (-Q_{ii})>0$, we construct a transition probability matrix  
\begin{equation}\label{KM}
K:=\frac{Q}{a}+I
\end{equation}
 such that
\begin{equation}
K_{ij} = \frac{Q_{ij}}{a}\geq 0, \,\,\, j\neq i; \quad K_{ii} = 1+ \frac{Q_{ii}}{a} > 0.
\end{equation}
Since $K_{ii}>0$, the resulted Markov chain with transition probability matrix $K$ is aperiodic. 
It is easy to see $\sum_j K_{ij} = 1$, so we know $1$ is an eigenvalue to $K$ with a right eigenvector $e:=(1,1, \cdots, 1)^T$, i.e. $Ke=e.$ Then by the Perron-Frobenius
theorem for $K$, we know 
\begin{itemize}
\item $\mu_1=1$ is the principle eigenvalue of $K$ and all the other eigenvalues $|\mu_i|<1$, $i=2, \cdots, n$;
\item  there is  a positive eigenvector, denoted as $\pi_\8^T$, being the left eigenvector corresponding to $1$, i.e. $\pi_\8^T K = \pi_\8^T$. 
\end{itemize}
Here and in the following context, the vector $\rho^T$ is short for $\rho^T:=(\rho_1|C_1|, \cdots, \rho_n|C_n|)$. 
Thus we know the positive vector $\pi^\8$  is the unique numerical  stationary solution to 
\begin{equation}\label{equiC}
\sum_j Q^*_{ij}\pi^\8_j|C_j| =0.
\end{equation} 

One can  interpret the right-hand-side of \eqref{rhoDeq} as the  flux gained and lost at    site $i$. Denote 
 \begin{equation}
 F_{ji} : =Q_{ji} \rho_j |C_j| - Q_{ij} \rho_i |C_i|,
 \end{equation}
 then  the dynamic solution satisfies
\begin{equation}\label{dynF}
\pt_t(\rho_i|C_i|)=\sum_{j\in VF(i)} F_{ji}, \quad i=1,2, \cdots, n.
\end{equation}
Denote the steady flux as
\begin{equation}\label{equiF}
F^\pi_{ji} : =Q_{ji} \pi^\8_j |C_j| - Q_{ij} \pi^\8_i |C_i|, \quad \sum_{j\in VF(i)} F_{ji}^\pi = 0 
\end{equation}
then we have  the symmetric property
\begin{equation}
F_{ij} = -F_{ji}, \quad F^\pi_{ij} = - F^\pi_{ji}.
\end{equation}

Next, leveraging the positive numerical steady state $\pi^\8$, we reformulate the flux $F_{ji}$ using a discrete $\pi$-symmetric decomposition. Using an elementary identity, we have 
\begin{equation}\label{decom}
\begin{aligned}
F_{ji}=& Q_{ji}\rho_j |C_j|- Q_{ij} \rho_i|C_i|\\
=&\frac{Q_{ji}\pi^\8_j |C_j| + Q_{ij} \pi^\8_i |C_i|}{2} \bbs{\frac{\rho_j}{\pi^\8_j} - \frac{\rho_i}{\pi^\8_i}} + \frac{Q_{ji} \pi^\8_j |C_j| - Q_{ij}\pi^\8_i |C_i|}{2} \bbs{\frac{\rho_j}{\pi^\8_j} + \frac{\rho_i}{\pi^\8_i}}\\
=& \frac12\alpha_{ij} \bbs{\frac{\rho_j}{\pi^\8_j} - \frac{\rho_i}{\pi^\8_i}}  + \frac12 F^\pi_{ji} \bbs{\frac{\rho_j}{\pi^\8_j} + \frac{\rho_i}{\pi^\8_i}}=:L_{ij}+T_{ij},
\end{aligned}
\end{equation}
where $\alpha_{ij} :=Q_{ji}\pi^\8_j |C_j| + Q_{ij} \pi^\8_i |C_i| =\alpha_{ji}>0$ is the symmetric positive coefficients. For the reversible case, $\alpha_{ij}$ is known as Onsager matrix. 

From the row sums zero property for $Q$-matrix, the natural tangent space of the probability space will be $T_P:=\{u\in \bR^n; \sum_i u_i = 0 \text{ and } u_i\geq 0 \text{ if } \rho_i=0\}$. Then from the antisymmetric property, we have
\begin{equation}
\sum_{ij} \alpha_{ij} \bbs{\frac{\rho_j}{\pi^\8_j} - \frac{\rho_i}{\pi^\8_i}} =0, \quad \sum_{ij}F^\pi_{ji} \bbs{\frac{\rho_j}{\pi^\8_j} + \frac{\rho_i}{\pi^\8_i}}=0.
\end{equation}
{ Thus (v) discrete $\pi$-symmetric decomposition \eqref{decom} holds.} It ensures each component in the decomposition still belongs to the tangent space $T_P.$

From stationary equation \eqref{equiF} and \eqref{sys},  we have
\begin{equation}\label{sy2}
\begin{aligned}
\sum_j T_{ij}&= \sum_j \frac{F^\pi_{ji}}{2} \frac{\rho_j}{\pi^\8_j}, \quad \sum_{i,j} T_{ij}  \frac{\rho_i}{\pi^\8_i} = \sum_{i,j} F_{ji}^\pi \frac{\rho_j}{\pi^\8_j} \frac{\rho_i}{\pi^\8_i}=0.
\end{aligned}
\end{equation}
Therefore the second part $T_{ij}$ in the decomposition has no contribution in the energy dissipation law. Indeed, if the reversible condition (detailed balance) holds, i.e., $Q_{ji} \pi^\8_j |C_j| = Q_{ij}\pi^\8_i |C_i|$, then $T_{ij}=0$ and we obtain exactly same energy dissipation law as the irreversible case. 
However, for the irreversible case, i.e., $Q_{ji} \pi^\8_j |C_j| \neq Q_{ij}\pi^\8_i |C_i|$, \textsc{Prigogine} characterized the energy transport when the dynamical system reaches the ``stationary non-equilibrium state $\pi$''  via the concept of entropy production rate  \cite{prigogine1968introduction}. Irreversible process is important for computing  transition paths in chemical reaction \cite{GLLL21} and constructing the global energy landscape \cite{gao2022selection}.   Importantly, using  Markov chain models in  biochemical reactions, \textsc{Hill} gives the following specific formula for the entropy production rate and call it ``the total rate of free energy dissipation''    \cite[eq. (9.20)]{hill2005free}
\begin{equation}\label{entropyP}
\ud S:= \frac{kT}2 \sum_{i,j} \bbs{Q_{ij}\pi^\8_i|C_i| - Q_{ji}\pi^\8_j|C_j|} \log \frac{Q_{ij} \pi^\8_i|C_i| }{Q_{ji}\pi^\8_j|C_j|} >0.
\end{equation}
Apparently, $\ud S=0$ if and only if the detailed balance condition holds (reversible case). See Appendix \ref{app:Hill} for detailed derivations by \textsc{Hill}.

In the following convergence analysis, since $D$ is positive definite, so $\vec{n}_{ij} \cdot D\vec{n}_{ij}\geq c>0$. Thus for notation simplicity, we now assume $D_{ij}=D\delta_{ij}, \, D>0$ is a scalar constant.
 From the definition of $Q$ for \eqref{mp}, 
$$F_{ji} = |\Gamma_{ij}| \bbs{\frac{D \bbs{\rho_j-\rho_i}}{|\my_j-\my_i|}  + (\vec{b}\cdot \vec{n})_{ij}^- \rho_j -  (\vec{b}\cdot \vec{n})_{ij}^+ \rho_i }.$$
Respectively,  the steady flux at site $i$  is
\begin{equation}
\begin{aligned}
F^\pi_{ji}= |\Gamma_{ij}| \bbs{\frac{D \bbs{\pi^\8_j-\pi^\8_i}}{|\my_j-\my_i|}  + (\vec{b}\cdot \vec{n})_{ij}^- \pi^\8_j -  (\vec{b}\cdot \vec{n})_{ij}^+ \pi^\8_i }.
\end{aligned}
\end{equation}

We summarize (vi) discrete energy dissipation law below.
\begin{prop}\label{prop_energy}
Let $\rho_i$ be the solution to \eqref{dynF} and $\pi^\8_i$ be the solution to \eqref{equiF}. 
\begin{equation}\label{Ddecay2}
\frac{\ud}{\ud t} \sum_i \frac{\rho_i^2}{\pi^\8_i}|C_i| = -\frac12 \sum_{i,j} \alpha_{ij} \bbs{\frac{\rho_j}{\pi^\8_j} - \frac{\rho_i}{\pi^\8_i}}^2 \leq -\frac{D}{2} \sum_{i,j}  \frac{|\Gamma_{ij}|(\pi^\8_i+\pi^\8_j)}{|\my_i-\my_j|} \bbs{\frac{\rho_j}{\pi^\8_j} - \frac{\rho_i}{\pi^\8_i}}^2 \leq 0,
\end{equation}
where $\alpha_{ij}=Q_{ji}\pi^\8_j |C_j| + Q_{ij} \pi^\8_i |C_i|=\alpha_{ji}>0$.
Moreover, we have also 
\begin{equation}\label{Ddecay3}
\frac{\ud}{\ud t} \sum_i \frac{(\rho_i-\pi^\8_i)^2}{\pi^\8_i}|C_i| = -\frac12 \sum_{i,j} \alpha_{ij} \bbs{\frac{\rho_j}{\pi^\8_j} - \frac{\rho_i}{\pi^\8_i}}^2 \leq -\frac{D}{2} \sum_{i,j}  \frac{|\Gamma_{ij}|(\pi^\8_i+\pi^\8_j)}{|\my_i-\my_j|} \bbs{\frac{\rho_j}{\pi^\8_j} - \frac{\rho_i}{\pi^\8_i}}^2 \leq 0,
\end{equation}
\end{prop}
We remark the $\chi^2$-divergence energy identity \eqref{Ddecay2} always holds with/without detailed balance condition. Usually the dissipation term $\frac12 \sum_{i,j} \alpha_{ij} \bbs{\frac{\rho_j}{\pi^\8_j} - \frac{\rho_i}{\pi^\8_i}}^2$ is called discrete Dirichlet form.
As a consequence, the symmetric coefficients for ``upwind scheme'' is
\begin{equation}\label{al}
\alpha_{ij} :=  |\Gamma_{ij}| \bbs{\frac{D}{|\my_j-\my_i|} \bbs{\pi^\8_i+\pi^\8_j} +  (\vec{b}\cdot \vec{n})_{ij}^+\pi^\8_i + (\vec{b}\cdot \vec{n})_{ji}^+\pi^\8_j  } =\alpha_{ji},
\end{equation}
 while for ``$\pi$-symmetric upwind scheme'' is
\begin{equation}
\alpha_{ij}= |\Gamma_{ij}| \bbs{\frac{D(\pi_i+\pi_j) }{|\my_j-\my_i|} + | (\vec{u}\cdot \vec{n})_{ij}| }=\alpha_{ji}.
\end{equation}
\begin{proof}

Using the definition of $Q$, direct calculations show $\alpha_{ij}=Q_{ji}\pi^\8_j |C_j| + Q_{ij} \pi^\8_i |C_i| $  is given by \eqref{al}.
From the decomposition \eqref{decom} and \eqref{sy2},  the discrete energy dissipation law becomes
\begin{equation}\label{Ddecay1}
\begin{aligned}
\sum_i \frac{\rho_i}{\pi^\8_i} \dot{\rho_i}|C_i| = \sum_i \bbs{\frac{\rho_i}{\pi^\8_i} \sum_j F_{ji}}
=-\frac{1}{4}\sum_{i,j} \alpha_{ij} \bbs{\frac{\rho_j}{\pi^\8_j} - \frac{\rho_i}{\pi^\8_i}}^2\leq 0.       
\end{aligned}
\end{equation}
Since $\pi^\8$ is positive, we have the estimate 
 \begin{equation}\label{al_error}
 \alpha_{ij} >  D \frac{|\Gamma_{ij}| (\pi^\8_i+\pi^\8_j)}{|\my_j-\my_i|}.
 \end{equation}
Then \eqref{Ddecay1} implies \eqref{Ddecay2}.
Moreover, from conservation of total mass, 
\begin{equation}
\sum_i \frac{\rho_i}{\pi^\8_i} \dot{\rho_i}|C_i|=\sum_i\bbs{ \frac{\rho_i}{\pi^\8_i}-1} \dot{\rho_i}|C_i| =\frac12 \frac{\ud}{\ud t} \sum_i \frac{(\rho_i-\pi^\8_i)^2}{\pi^\8_i}|C_i|,
\end{equation}
so we also conclude \eqref{Ddecay3}.
\end{proof}

We point out that  the stability  analysis here and in the remaining section   also works for the Fokker-Planck equation on non-compact manifold as long as a  positive solution to $Q^* \pi^\8=0$ exists. However, there is no general  Perron-Frobenius theorem to ensure  the existence of \eqref{equiC} for a $Q$-matrix on a countable states.

\begin{rem}
Compared with the $\pi$-symmetric decomposition in the continuous case \eqref{FPr}, the gradient flow part in \eqref{decom} 
\begin{equation}
L_{ij} = \frac{D|\Gamma_{ij}|}{|\my_j-\my_i|}  \frac{\pi^\8_i+\pi^\8_j}{2}\bbs{\frac{\rho_j}{\pi^\8_j} - \frac{\rho_i}{\pi^\8_i}}  \bbs{ 1+   |\my_j-\my_i| \frac{{(\vec{b}\cdot \vec{n})_{ij}^+\pi^\8_i + (\vec{b}\cdot \vec{n})_{ji}^+\pi^\8_j } }{D(\pi^\8_i+ \pi^\8_j)}}
\end{equation}  has an additional numerical dissipation terms $O\bbs{|\my_j-\my_i|}$. This numerical dissipation comes from the upwind discretization. Indeed, without this additional numerical dissipation term,   $\pt_t \rho_i|C_i| = \sum_{j}L_{ij} $ is exactly the  finite volume scheme designed in \cite{GLW20} for reversible case.
\end{rem}

 We remark that as long as we have $Q$-matrix structure, the $\phi$-divergence dissipation holds. Indeed, let $\phi$ be a convex function and $E=\sum_i \phi\bbs{\frac{\rho_i}{\pi_i}}\pi_i $. Then using $\sum_j Q_{ji}\pi_j = 0$ and $\sum_j Q_{ij}=0$ for any $i$, we have
\begin{equation}
\begin{aligned}
\frac{\ud}{\ud t} \sum_i \rho_i \phi'\bbs {\frac{\rho_i}{\pi _i}} =& \sum_{i,j}  Q_{ji} \rho_j \phi'\bbs{\frac{\rho_i}{\pi_i}} = \sum_{ i, j, i\neq j} Q_{ji} \pi_j \frac{\rho_j}{\pi_j} \bbs{\phi' \bbs{\frac{\rho_i}{\pi_i}}-\phi'\bbs{\frac{\rho_j}{\pi_j}}}\\
 =&\sum_{i,j, j\neq i} Q_{ji} \pi_j \frac{\rho_j}{\pi_j} \bbs{\phi' \bbs{\frac{\rho_i}{\pi_i}}-\phi'\bbs{\frac{\rho_j}{\pi_j}}} -  \sum_{i,j, j\neq i} Q_{ji} \pi_j \bbs{\psi\bbs{\frac{\rho_i}{\pi_i}} -\psi\bbs{\frac{\rho_j}{\pi_j}} }\\
  =&\sum_{i,j, j\neq i} Q_{ji} \pi_j \bbs{ \frac{\rho_j}{\pi_j} \bbs{\phi' \bbs{\frac{\rho_i}{\pi_i}}-\phi'\bbs{\frac{\rho_j}{\pi_j}}} -  \bbs{\psi\bbs{\frac{\rho_i}{\pi_i}} -\psi\bbs{\frac{\rho_j}{\pi_j}} } },
\end{aligned}
\end{equation}
where $\psi(x)$ is any function to be chosen later. Denote $y=\frac{\rho_j}{\pi_j}$ and $x=\frac{\rho_i}{\pi_i}$. To compute the Bregman divergence $D_\phi(y,x)$ associated with $\phi$ for points $y,x$,
take $\psi(x) = x\phi'(x)-\phi(x)$, then $\psi'(x)=x\phi''(x)$
and
\begin{equation}
y [\phi'(x) - \phi'(y)] - [\psi(x) - \psi(y)] = (y-x)\phi'(x) +\phi(x) - \phi(y)=: - D_\phi(y,x)
\end{equation}
Using the integral form of the reminder in Taylor expansion, 
\begin{equation}
 D_\phi(y,x)= (y-x)^2\int_0^1 (1-\theta)\phi''(x+\theta(y-x))\ud \theta \geq 0.
\end{equation}
 Then the 
dissipation relation becomes
\begin{equation}
\frac{\ud E}{\ud t}= -\sum_{i,j, j\neq i} Q_{ji} \pi_j D_\phi\bbs{ \frac{\rho_j}{\pi_j} , \frac{\rho_i}{\pi_i}} \leq 0.
\end{equation}

%
%

\subsection{Ergodicity}

Let $Q$ be the $Q$-matrix defined in \eqref{Qm}. Recall the  nonnegative irreducible stochastic matrix  $K$ defined in \eqref{KM}.
Recall the right eigenvector $e:=(1,1, \cdots, 1)^T$, i.e. $Ke=e$ corresponding to $1$ and 
the left eigenvector, i.e. $\pi^T K = \pi^T$  due to the Perron-Frobenius
theorem for $K$.
We have the Jordan decomposition for $K$
\begin{equation}\label{jordan}
K = S \left(\begin{array}{cc}
1 & 0 \\ 
0 & J
\end{array}  \right) S^{-1} = e \pi^T + S \left(\begin{array}{cc}
0 & 0 \\ 
0 & J
\end{array}  \right) S^{-1} 
\end{equation} 
where $J$ consists of Jordan blocks with $|\mu_i|<1$, $i=2, \cdots, n$  and we used $S_{i1}=e_i$, $S^{-1}_{1j}=\pi_j$.
Here and in the following context, the vector $\rho^T$ is short for $\rho^T:=(\rho_1|C_1|, \cdots, \rho_n|C_n|)$. 

\begin{lem}\label{lem_longtime}
Let $Q$ be the $Q$-matrix defined in \eqref{Qm} and $\rho_t$ be the solution to \eqref{rhoDeq}. Given  initial data $\rho_0$,   we have
\begin{equation}\label{longtime}
\|\rho_t - \pi\|_{\ell^1} \leq c e^{ \frac{at}{2}(|\mu_2|-1)} \|\rho_0 \|_{\ell^1},
\end{equation}
where $|\mu_2|<1$ is the second  eigenvalue of $Q$ and $a>0$ is the constant in \eqref{KM}.
\end{lem}
\begin{proof}
From the construction for $K$ and the Jordan decomposition in \eqref{jordan}, we know
\begin{equation}
Q = a(K-I) =  S \left(\begin{array}{cc}
0 & 0 \\ 
0 & a(J-I)
\end{array}  \right) S^{-1}.
\end{equation}
Thus 
\begin{equation}
e^{Qt} = e \pi^T + S \left(\begin{array}{cc}
0 & 0 \\ 
0 & e^{a(J-I)t}
\end{array}  \right) S^{-1}.
\end{equation}
Then for the solution $\rho_t$ to \eqref{rhoDeq}, we have
\begin{equation}\label{Drho_t}
\rho_t^T - \pi^T = \rho_0^T e^{Qt} - \pi^T =  \rho_0^T S \left(\begin{array}{cc}
0 & 0 \\ 
0 & e^{a(J-I)t}
\end{array}  \right) S^{-1}.
\end{equation}
Notice  $J$ consists of Jordan blocks with $|\mu_i|<1$, $i=2, \cdots, n$, and each Jordan block  $\ell$ can be written as the sum of $\mu_\ell I$ and a nilpotent matrix $N$. Thus   we have $e^{at(J_\ell-I)} = e^{at(\mu_\ell -1)}e^{atN}$ and $\|e^{atN}\|_{\ell^1}\leq ct^{n-2}$.   Therefore, combining  \eqref{Drho_t} with
$$\|e^{a(J-I)t}\|_{\ell^1}\leq c t^{n-2} e^{at(|\mu_2|-1)}\leq c e^{\frac{at(|\mu_2|-1)}{2} },$$
we conclude \eqref{longtime}.
\end{proof}

Another proof for the ergodicity is using the following discrete mean Poincare's inequality \cite[Lemma 10.2]{eymard2000finite}. As a consequence, the exponential decay rate will not depend on the data size $n$. 
\begin{lem}[\cite{eymard2000finite}]\label{lem_poin}
Assume $\sum_i u_i \pi_i |C_i| =0$, then we have the discrete mean Poincare inequality
\begin{equation}
    \sum_i u_i^2 \pi_i |C_i| \leq c  \sum_{i,j}  \frac{|\Gamma_{ij}|(\pi_i+\pi_j)}{|\my_i-\my_j|} \bbs{u_i - u_j}^2. 
\end{equation}
\end{lem}
Using \eqref{Ddecay3} in Proposition \ref{prop_energy} and then taking $u_i=\frac{\rho_i}{\pi_i} -1$ in Lemma \ref{lem_poin},  we have
\begin{equation}
\begin{aligned}
\frac{\ud}{\ud t} \sum_i \frac{(\rho_i-\pi_i)^2}{\pi_i}|C_i|   \leq -\frac{D}{2} \sum_{i,j}  \frac{|\Gamma_{ij}|(\pi_i+\pi_j)}{|\my_i-\my_j|} \bbs{\frac{\rho_j}{\pi_j} - \frac{\rho_i}{\pi_i}}^2\leq - c \sum_i \frac{(\rho_i-\pi)^2}{ \pi_i} |C_i|.
\end{aligned}
\end{equation}
This implies the exponential decay in $\chi^2$-divergence
\begin{equation}
\sum_i \frac{(\rho_i-\pi_i)^2}{\pi_i}|C_i| \leq  e^{-ct} \,  \sum_i \frac{(\rho_i(0)-\pi_i)^2}{\pi_i}|C_i|.
\end{equation}

\subsection{{ Unconditionally stable implicit scheme solved explicitly}}
Although there is no detailed balance property, we will use a mixed explicit-implicit time discretization to design an unconditionally stable mixed explicit-implicit scheme for \eqref{rhoDeq}, which enjoys a new stochastic-matrix structure. { Since this scheme can be solved explicitly, we will also call it as an explicit scheme for simplicity.}

Let $\rho_i^k|C_i|$ be the discrete density at the discrete time $k\Delta t$. Recall the constant $a>0$ and nonnegative irreducible matrix $K=\frac{Q}{a}+I$ defined in \eqref{KM}. Introduce the mixed explicit-implicit scheme as
\begin{equation}\label{time1}
\rho^{k+1}_i |C_i| - \rho^k_i |C_i| = \sum_{j} a\Delta t \bbs{ K_{ji}\rho_j^k |C_j| -  \rho_i^{k+1} |C_i|},\quad i=1, 2, \cdots, n.
\end{equation}

Now we recast it as a new discrete Markov chain with a transition probability $\tilde{K}$, which is unconditionally stable and enjoys fast convergence to the same steady state $\pi$. Rewrite \eqref{time1} as
\begin{equation}\label{num355}
(1+a \Delta t) \rho^{k+1}_i|C_i| = a \Delta t \sum_j K_{ji}\rho_j^k |C_j| + \rho^k_i|C_i| = \Delta t \sum_j Q_{ji} \rho^k_j |C_j| + (1+a \Delta t) \rho_i^k |C_i|,
\end{equation}
then we obtain a new Markov semigroup
\begin{equation}\label{timeG}
\rho^{k+1}_i |C_i| = \frac{\Delta t}{1+ a \Delta t} \sum_j Q_{ji} \rho^k_j |C_j| + \rho_i^k |C_i| =: \sum_j \tilde{K}_{ji} \rho_j^k |C_j|
\end{equation}
It is easy to verify 
\begin{equation}
 \tilde{K}= \Delta t \frac{Q}{1+a \Delta t} + I =: \Delta t \tilde{Q} + I 
\end{equation}
is a Markov semigroup operator  and $\pi$ being its positive invariant measure. Indeed,  
since 
\begin{equation}\label{tK}
\tilde{K} = \frac{a \Delta t}{1+a \Delta t} \frac{Q}{a} + I, \quad  \frac{a \Delta t}{1+a \Delta t}<1, 
\end{equation}
 we know $\tilde{K}$ is always nonnegative, irreducible, aperiodic matrix and obtain  the right eigenvector $e:=(1,1, \cdots, 1)^T$, i.e. $\tilde{K} e=e$ corresponding to $1$ and 
the left eigenvector, i.e. $\pi^T \tilde{K} = \pi^T$. 
From the Perron-Frobenius
theorem, we have the Jordan decomposition for $\tilde{K}$
\begin{equation}\label{jordan1}
\tilde{K} = S \left(\begin{array}{cc}
1 & 0 \\ 
0 & J
\end{array}  \right) S^{-1} = e \pi^T + S \left(\begin{array}{cc}
0 & 0 \\ 
0 & J
\end{array}  \right) S^{-1} ,
\end{equation} 
where $J$ consists of Jordan blocks with $\tilde{K}$'s eigenvalues $|\mu_i|<1$, $i=2, \cdots, n$  and we used $S_{i1}=e_i$, $S^{-1}_{1j}=\pi_j$.

Therefore, thanks to $\tilde{Q}$-matrix, the explicit scheme \eqref{timeG} still enjoys good properties such as positive preserving, total mass  preserving. Particularly, we give the following proposition for the unconditional  $\ell^1$-contraction and ergodicity.
\begin{prop}[{Stability and ergodicity for fully discretized scheme}]\label{prop_timeD}
Let $\Delta t$ be the time step and let $\rho^k$ be the solution to the explicit scheme \eqref{timeG}. Then we have
\begin{enumerate}[(i)]
 \item the $\ell^1$ contraction, i.e., for any two solutions $\rho^k, \tilde{\rho}^k$,
\begin{equation}\label{l8semi}
 \sum_i  |\rho_i^{k+1} - \tilde{\rho}_i^{k+1}| |C_i| \leq \sum_i   |\rho_i^{n} - \tilde{\rho}_i^{n}| |C_i|;
\end{equation}
\item the exponential convergence
\begin{equation}\label{gap_error}
\|\rho^k - \pi\|_{\ell^1}  \leq c \|\rho_0\|_{\ell^1} |\mu_2|^{k-n+2}, \quad |\mu_2|<1,
\end{equation}
where $\mu_2$ is the second eigenvalue (in terms of the magnitude) of $\tilde{K}= I+\Delta t \frac{Q}{1+a \Delta t}$.
\end{enumerate}
\end{prop}
\begin{proof}
First, let $\rho$ and $\tilde{\rho}$ be two solutions to \eqref{timeG}. Denote $e^k_i=\rho^k_i-\tilde{\rho}^k_i$. Then  \eqref{timeG} becomes
\begin{equation}
e^{k+1}_i |C_i| = \frac{\Delta t}{1+ a \Delta t} \sum_j Q_{ij} e^k_j |C_j| + e_i^k |C_i| = \frac{\Delta t}{1+ a \Delta t} \sum_{j\neq i} Q_{ij} e^k_j |C_j| + \bbs{1+\frac{Q_{ii}\Delta t}{1+ a \Delta t}} e_i^k |C_i|.
\end{equation}
From \eqref{tK}, each term above is nonnegative. Then take absolute value and summation in $i$ to show that
\begin{equation}
\begin{aligned}
\sum_i |e^{n+1}_i||C_i|  \leq&   \frac{\Delta t}{1+ a \Delta t} \sum_{i,j} Q_{ij} |e^k_j| |C_j| + \sum_i |e_i^k| |C_i|= \sum_i |e_i^k| |C_i|.
\end{aligned}
\end{equation}

Second, 
\eqref{timeG} and the Jordan decomposition \eqref{jordan1} yield
\begin{equation}
(\rho^{k} - \pi)^T = \rho_0^T \tilde{K}^k - \pi^T  = \rho^T_0 S \left(\begin{array}{cc}
0 & 0 \\ 
0 & J^n
\end{array}  \right) S^{-1}. 
\end{equation} 
Then from $\|J^n\|_{\ell^1} \leq c|\mu_2|^{k-n+2}$ and $\tilde{K}$'s  second eigenvalue $|\mu_2|<1$, we conclude
\begin{equation}
\|\rho^k - \pi\|_{\ell^1}  \leq c \|\rho_0\|_{\ell^1} |\mu_2|^{k-n+2}.
\end{equation}
\end{proof}

 The advantage of this mixed implicit-explicit scheme \eqref{time1} is that the Markov semigroup is always positive no matter how larger $\Delta t$ is. However, when $\Delta t$ becomes large, the spectral gap $1-|\mu_2|$ becomes small.

\section{Convergence and error estimates}\label{sec4}

In this section, we present  convergence and error estimates for the upwind scheme \eqref{mp} including  error estimates for numerical steady state solution $\pi^\8_i$ and the numerical dynamic solution $\rho_i(t)$. The estimates relies on the Taylor expansion along geodesic between each cell center and the discrete energy dissipation in Proposition \ref{prop_energy}. Same procedures can be done for $\pi$-symmetric upwind scheme and will be omitted.
\subsection{Error estimate for steady state $\pi$}
Let $\pi(\my)$ be the exact solution to stationary Fokker-Planck equation \eqref{FPequi} and let $\pi_i$ be the numerical solution satisfying \eqref{equiF}. In this section, for notation simplicity, we denote $\pi_i$ as the numerical steady state $\pi^\8_i$.
Plug the exact solution $\pi(\my)$  into the numerical scheme and we estimate 
\begin{equation}
\eps^\pi_{ji} :=  \int_{\Gamma_{ij}} \mathbf{n} \cdot \bbs{D \nabla \pi - \vec{b} \pi} \ud \hs^{d-1}  ~ - ~ F^{\pi_{\text{ex}}}_{ji},
\end{equation}
where $\mathbf{n}$ is the restriction of the unit outward normal vector field on $\Gamma_{ij}$.
Let $e_i:=\pi_i-\pi(\my_i)$ be the error between   the numerical steady state and the exact steady state. Then we have
\begin{equation}\label{errorEqui}
\begin{aligned}
&F^e_{ji}:=F^{\pi}_{ji}- F^{\pi_{\text{ex}}}_{ji}=|\Gamma_{ij}| \bbs{\frac{D \bbs{e_j-e_i}}{|\my_j-\my_i|}  + (\vec{b}\cdot \vec{n})_{ij}^- e_j -  (\vec{b}\cdot \vec{n})_{ij}^+ e_i }= -F^e_{ij},\\
&\sum_j F^e_{ji}  =\sum_j  \eps^{\pi}_{ji}.
\end{aligned}
\end{equation}
\begin{prop}\label{prop_pi_con}
Let $e_i:=\pi_i-\pi(\my_i)$ be the error between   the numerical steady state $\pi_i$ and the exact steady state $\pi(\my_i)$. Assume the Vonoroi tessellation satisfies
\begin{equation}\label{asm_V}
\sum_{i}\sum_{j\in VF(i)} |\Gamma_{ij}| |\my_j - \my_i| \leq c.
\end{equation}
Then we have
\begin{equation}\label{Pi_con}
\sum_{i} \frac{e_i^2}{\pi_i} |C_i| \leq c \sum_{i,j} \frac{|\Gamma_{ij}|}{|\my_j-\my_i|} \bbs{\frac{e_j}{\pi_j} - \frac{e_i}{\pi_i}}^2  \leq c h^2.
\end{equation}
\end{prop}

\begin{proof}
Step 1. Estimates for truncation error $\eps_{ji}^\pi$.

Let $G_{ij}$ be the bisector between $\my_i$ and $\my_j$, which is a $d-1$ submanifold containing $\Gamma_{ij}$. Suppose $\my^*$ is the intersection point of the  geodesic from $\my_i$ to $\my_j$ and $G_{ij}$. We have $d_{\nn}(\my^*, \my_i)=d_{\nn}(\my^*, \my_j)$.  Notice the unit tangent vector of the  geodesic at $\my^*$ is perpendicular to $\Gamma_{ij}$ at $\my^*$ and can be chosen as the unit normal $\vec{n}_{ij}(\my^*)$\footnote{on the $d-1$ dimensional submanifold containing $\Gamma_{ij}$}.   From the Taylor expansion of $\pi$ along the geodesic, we have
\begin{align}
&\pi(\my_j)-\pi(\my^*)=\vec{n}_{ij} \cdot \nabla \pi(\my^*)d_{\nn}(\my^*, \my_j)+O(d^2_{\nn}(\my^*, \my_j)), \\
& {\pi}(\my^*)- {\pi}(\my_i)=\vec{n}_{ij} \cdot \nabla \pi(\my^*)d_{\nn}(\my^*, \my_i)+O(d^2_{\nn}(\my^*, \my_i)). 
\end{align}
  Therefore, if we add the above two equations , we have
\begin{align}
& {\pi}(\my_j)- {\pi}(\my_i)=\mathbf{n}_{ij} \cdot \nabla  {\pi}(\my^*) d_{\nn}(\my_i, \my_j)+O(d^2_{\nn}(\my_i, \my_j)).
\end{align}
Hence,
\begin{align}\label{tm1}
\mathbf{n}_{ij} \cdot \nabla  {\pi}(\my^*)=\frac{ {\pi}(\my_j)- {\pi}(\my_i)}{ d_{\nn}(\my_i, \my_j)}+O(d_{\nn}(\my_i, \my_j))=\frac{{\pi}(\my_j)- {\pi}(\my_i) }{|\my_i-\my_j|}+O(d_{\nn}(\my_i, \my_j)),
\end{align}
where we used $|\my_i-\my_j| = d_{\nn}(\my_i, \my_j) + o(d_{\nn}(\my_i, \my_j)) $ for $d_{\nn}(\my_i, \my_j)$ small enough. Similarly,
\begin{align}
&\pi(\my_j)-\pi(\my^*)=O(d_{\nn}(\my^*, \my_j)), \\
&\pi(\my^*)-\pi(\my_i)=O(d_{\nn}(\my^*, \my_i)). 
\end{align}
Hence we have
\begin{align}\label{tm2}
\bbs{ -(\vec{b}\cdot \vec{n})^+_{ij} + (\vec{b}\cdot \vec{n})^-_{ij}  }\pi(\my^*)= -(\vec{b}\cdot \vec{n})^+_{ij} \pi(\my_i) + (\vec{b}\cdot \vec{n})^-_{ij}  \pi(\my_j) +O(d_{\nn}(\my_i, \my_j)).
\end{align}
Thus if $\my^*$ is changes to any $\my$ on $\Gamma_{ij}$, 
$$O(d_{\nn}(\my_i,\my)+d_{\nn}(\my_i,\my_j)) = O(\diam(C_i)+d_{\nn}(\my_i,\my_j)) \leq O(h),
$$
where
\begin{align}\label{hh}
h:=\max\big( \max_{i=1, \ldots, n}(\diam(C_i)), \max_{i=1, \ldots, n}(\max_{j\in VF(i)} d_{\nn}(\my_i,\my_j))\big).
\end{align}
This, together with \eqref{tm1} and \eqref{tm2}, we conclude
\begin{equation}\label{Terror}
\eps^\pi_{ij} \leq c h |\Gamma_{ij}|.
\end{equation}

Step 2. $H^1$ estimates via anti-symmetric structure.

Using same derivations as Proposition \ref{prop_energy}, we have
\begin{equation}\label{DecayPi}
\begin{aligned}
 \sum_{i} \bbs{\frac{e_i}{\pi_i} \sum_j F_{ji}^e } = -\frac12 \sum_{i,j} \alpha_{ij} \bbs{\frac{e_j}{\pi_j} - \frac{e_i}{\pi_i}}^2,
\end{aligned}
\end{equation}
where $\alpha_{ij}$ defined in \eqref{al}.
And thus \eqref{errorEqui} implies
\begin{equation}
\begin{aligned}
 \sum_{i} \bbs{\frac{e_i}{\pi_i} \sum_j F_{ji}^e } = \sum_{i,j}\eps^\pi_{ji} \frac{e_i}{\pi_i} =  \frac12\sum_{i,j} \eps^\pi_{ji} \bbs{\frac{e_i}{\pi_i} - \frac{e_j}{\pi_j}}. 
\end{aligned}
\end{equation}
Then combining this with \eqref{DecayPi} and  Young's inequality,  we know
\begin{equation}
\begin{aligned}
  \sum_{i,j} \alpha_{ij} \bbs{\frac{e_j}{\pi_j} - \frac{e_i}{\pi_i}}^2  = -\sum_{i,j} \eps^\pi_{ji} \bbs{\frac{e_i}{\pi_i} - \frac{e_j}{\pi_j}}\leq \frac12 \sum_{i,j} \alpha_{ij} \bbs{\frac{e_i}{\pi_i} - \frac{e_j}{\pi_j}}^2 + \frac12\sum_{i,j}\frac{(\eps_{ji}^\pi)^2}{\alpha_{ij}}.
  \end{aligned}
\end{equation}
Then by \eqref{Terror} and \eqref{al_error}, we conclude
\begin{equation}
 \sum_{i,j} \alpha_{ij} \bbs{\frac{e_j}{\pi_j} - \frac{e_i}{\pi_i}}^2  \leq \sum_{i,j}\frac{(\eps_{ji}^\pi)^2}{\alpha_{ij}} \leq c h^2 \sum_{i,j}   |\Gamma_{ij}||\my_j-\my_i| \leq c h^2.
\end{equation}
Moreover,  since $\sum_i e_i|C_i|=0$, we take $u_i=\frac{e_i}{\pi_i}$ in Lemma \ref{lem_poin} to obtain
\begin{equation}
\sum_{i} \frac{e_i^2}{\pi_i} |C_i| \leq c \sum_{i,j} \alpha_{ij} \bbs{\frac{e_j}{\pi_j} - \frac{e_i}{\pi_i}}^2 \leq ch^2.
\end{equation}
This, together with the definition of $\alpha_{ij}$ in \eqref{al}, concludes \eqref{Pi_con}. 
 \end{proof}

\subsection{Convergence of dynamic solution in $\chi^2$-discrepancy}
Let $\rho_t(\my)$ be the exact solution to \eqref{FP-N}. Let $e_i(t) := \rho_t(\my_i)-\rho_i(t)$ be the error between the exact solution and the numerical solution. Now for any fixed $T>0$, we give the convergence result in terms of the $\chi^2$-discrepancy between the  exact solution and the numerical solution.
\begin{thm}\label{thm_con}
Let $e_i(t):=\rho_i(t)-\rho_t(\my_i)$ be the error between   the numerical solution $\rho_i(t)$ and the exact steady state $\rho_t(\my_i)$. Assume the Vonoroi tessellation satisfies \eqref{asm_V},
then for any $T>0$ and $h$ defined in \eqref{hh}, we have 
\begin{equation}
\max_{t\in[0,T]} \sum_i  \frac{e_i(t)^2}{\pi_i}|C_i| \leq \bbs{ \sum_i  \frac{e_i(0)^2}{\pi_i}|C_i| +O( h^2) } e^T .
\end{equation}
\end{thm}
\begin{proof}
Let  $\rho_i^e:= \frac{1}{|C_i|}\int_{C_i} \rho \ud y$ be the cell average. Plug the exact solution into the numerical scheme
\begin{equation}\label{mp-exact}
\begin{aligned}
\pt_t (\rho_i^e |C_i|) &=  \sum_{j\in VF(i)} F^{ex}_{ji}  + \sum_{j\in VF(i)}\eps_{ji},\\
&\eps_{ji}:= \int_{\Gamma_{ij}} \mathbf{n}  \cdot \bbs{D \nabla \rho - \vec{b} \rho} \ud \hs^{d-1}  ~ - ~ F^{ex}_{ji},
\end{aligned}
\end{equation}
where $\mathbf{n} $ is the restriction of the unit outward normal vector field on $\Gamma_{ij}$ and 
\begin{equation}
F^{ex}_{ji}:=\frac{1}{2} \alpha_{ij} \bbs{\frac{\rho(\my_j)}{\pi_j} - \frac{\rho(\my_i)}{\pi_i}}      +\frac{F^\pi_{ji}}{2} \bbs{\frac{\rho(\my_i)}{\pi_i}+\frac{\rho(\my_j)}{\pi_j}}
\end{equation}
with $\alpha_{ij}$ defined in \eqref{al}.
 Exchanging $i,j$ above, we  see both $\eps_{ji}$ and $F_{ji}^{ex}$ are  anti-symmetric.

Subtracting the numerical scheme \eqref{mp} from \eqref{mp-exact}, 
we have
\begin{equation}
\frac{\ud}{\ud t} e_i|C_i|  =  \sum_{j\in VF(i)} \bbs{ \frac{\alpha_{ij}}{2} \left(  \frac{e_j}{\pi_j}- \frac{e_i}{\pi_i}\right) + \frac{F^\pi_{ji}}{2} \bbs{\frac{e_i}{\pi_i}+\frac{e_j}{\pi_j}}} + \sum_{j\in VF(i)} \eps_{ji}+  \pt_t  (\rho(\my_i) -\rho_i^e)|C_i| .
\end{equation}
Similar to the derivation of dissipation relation \eqref{Ddecay2}, we multiply this by  $\frac{2e_i}{\pi_i}$ and use \eqref{equiF} to show that
\begin{equation}
 \frac{\ud}{\ud t} \sum_i \frac{e_i^2}{\pi_i}|C_i| = -\sum_{i,j} \frac{ \alpha_{ij}}{2}   \left(  \frac{e_j}{\pi_j}- \frac{e_i}{\pi_i}\right)^2 \,+ \sum_{i,j}   \eps_{ji} \left(\frac{e_i}{\pi_i} - \frac{e_j}{\pi_j} \right) +\sum_i 2 \pt_t(\rho(\my_i) -\rho_i^e)  |C_i| \frac{e_i}{\pi_i},
\end{equation}
due to $\eps_{ij}=-\eps_{ji}$.
Applying  Young's inequality to the last two terms, we have
\begin{equation}
\begin{aligned}
&\sum_{i,j} \eps_{ji} \left(\frac{e_i}{\pi_i} - \frac{e_j}{\pi_j} \right)
\leq  \sum_{i,j}  \frac{\alpha_{ij}}{4}  \left(  \frac{e_j}{\pi_j}- \frac{e_i}{\pi_i}\right)^2 +  \sum_{i} \sum_{j\in VF(i)} \frac{\eps^2_{ji}}{\alpha_{ij}};\\
&\sum_i 2 \pt_t(\rho(\my_i) -\rho_i^e)|C_i|   \frac{ e_i}{\pi_i}\leq \sum_i  [\pt_t(\rho(\my_i) -\rho_i^e)]^2  \frac{ |C_i|}{\pi_i}  + \sum_i   \frac{  e_i^2}{\pi_i}|C_i|.
\end{aligned}
\end{equation}
Thus we have
\begin{align}
 \frac{\ud}{\ud t} \sum_i  \frac{e_i(t)^2}{\pi_i} |C_i|
\leq &  \sum_{i,j}  \frac{\eps^2_{ji}}{\alpha_{ij}} +\sum_i  [\pt_t(\rho(\my_i) -\rho_i^e)]^2  \frac{ |C_i|}{\pi_i}  + \sum_i   \frac{ e_i^2 }{\pi_i}|C_i|.
\end{align}

The estimates for $\eps_{ji}$ and $\pt_t(\rho(\my_i) -\rho_i^e)$ using Taylor expansion on manifold  are same as \eqref{Terror}.
From \eqref{Terror} and \eqref{al_error}, we have
\begin{equation}
\frac{\eps^2_{ji}}{\alpha_{ij}} \leq c h^2 |\my_j - \my_i||\Gamma_{ij}|.
\end{equation}
From the assumption \eqref{asm_V},
we know 
\begin{equation}
\sum_{i,j}\frac{\eps^2_{ji}}{\alpha_{ij}} \leq c h^2.
\end{equation}
 By Gronwall's inequality, we obtain for any $T>0$,
\begin{equation}
\max_{t\in[0,T]} \sum_i  \frac{e_i(t)^2}{\pi_i}|C_i| \leq \big (\sum_i  \frac{e_i(0)^2}{\pi_i}|C_i| +O( h^2) \big)  e^T .
\end{equation}
\end{proof}
{ We remark that for numerical solution itself, the  exponential convergence to the numerical invariant measure (which recovers exact invariant measure) is proved in Section \ref{sec3}. However, the error between the exact solution and the numerical solution still depends on $e^T$ as long as one use Gronwall's inequality in the error estimate. With the well-balanced property in \eqref{dis_db}, it is possible to obtain a uniform error estimate but we leave it for the future study.}

\section{Numerical examples}\label{sec5} 

In this section, we demonstrate three examples based on two upwind schemes developed in Section \ref{sec2}. Section \ref{sec_5.1pi} focus on the case that we know  information of an invariant measure $\pi$. In realistic situations,  knowing the output of $\pi$ from a computer code instead of  an  analytic formula is enough.  In this case, we use $\pi$-symmetric upwind scheme \eqref{mp-pi} to conduct two interesting applications: (i)   efficient sampling enhanced by an incompressible mixture flow (ii) image transformations immersed in a mixture flow. Section \ref{sec_5.2van} focus on the case we only know  the  drift vector field $\vec{b}$ in the irreversible drift-diffusion process \eqref{sde-y} and we adapt the numerical scheme \eqref{mp}  to solve both a steady state solution and simulate the irreversible dynamics. The  $\pi$-symmetric upwind scheme  \eqref{num2d_pi} and  upwind scheme \eqref{Van_scheme} for a 2D structured grids case with no-flux boundary conditions are given in Appendix \ref{app1} and Appendix \ref{app2} for completeness.

\subsection{Simulations for  irreversible dynamics given an invariant measure}\label{sec_5.1pi}
In this section, assume we know  the output of an invariant measure $\pi$, which could be a steady state in a biochemical process, or a given target density function in sampling, or  given images. We will use $\pi$-symmetric upwind scheme \eqref{mp-pi} for a 2D structured grids to demonstrate two examples.  

\subsubsection{Example: Sampling  enhanced by an incompressible mixture flow}
Take 2D domain as $\Omega=[a,b]\times[c,d]=[-4.5,4.5]\times[-4.5,4.5]$.
Choose the   stream function for a 2D sinusoidal cellular flow
\begin{equation}\label{psi}
\psi(x,y):= A\sin \frac{k\pi(x-a)}{b-a} \sin \frac{k\pi(y-c)}{d-c},
\end{equation} 
where $A$ represents the amplitude of the mixture velocity $\vec{u}$ and $k$ is the normalized wave number  of the mixture.
Then the incompressible velocity field 
$
\vec{u} = \bbs{\begin{array}{cc}
-\pt_y \psi\\
\pt_x \psi
\end{array}
} $
can be discretized using \eqref{streamU}.

Now we use \eqref{F_num}, i.e., \eqref{num2d_pi} for 2D structured grids, to sample a given target density: a smiling triple-banana in \eqref{smile}. Based on the Laplace principle, we can use the smooth minimum method to construct a smiling triple-banana as a target density
 \begin{equation}\label{smile}
 \begin{aligned}
 \pi(x,y) \propto &e^{ -20\big[\bbs{x-\frac65}^2+\bbs{y-\frac{6}{5}}^2 - \frac12\big]^2  + \log\bbs{e^{-10(y-2)^2}} } +  e^{-20 \big[\bbs{x+\frac65}^2+\bbs{y-\frac{6}{5}}^2 - \frac12\big]^2 + \log\bbs{e^{-10(y-2)^2}}} \\
 &+ e^{-20 \bbs{x^2+y^2 - 2}^2 + \log\bbs{e^{-10(y+1)^2}}}+0.1;
 \end{aligned}
 \end{equation}
 see Fig.\ref{fig_sample1}(right).
Then we take a  Gaussian mixture
\begin{equation}
\rho_0(x,y)\propto e^{-16(x+3)^2-4y^2}+ e^{-16(x-3)^2-4y^2}+ e^{-4x^2-16(y+3)^2}+ e^{-4x^2-16(y-3)^2}+0.1
\end{equation}
as an initial density; see Fig.\ref{fig_sample1}(left).  
\begin{figure}
\begin{center}
 \includegraphics[scale=0.3]{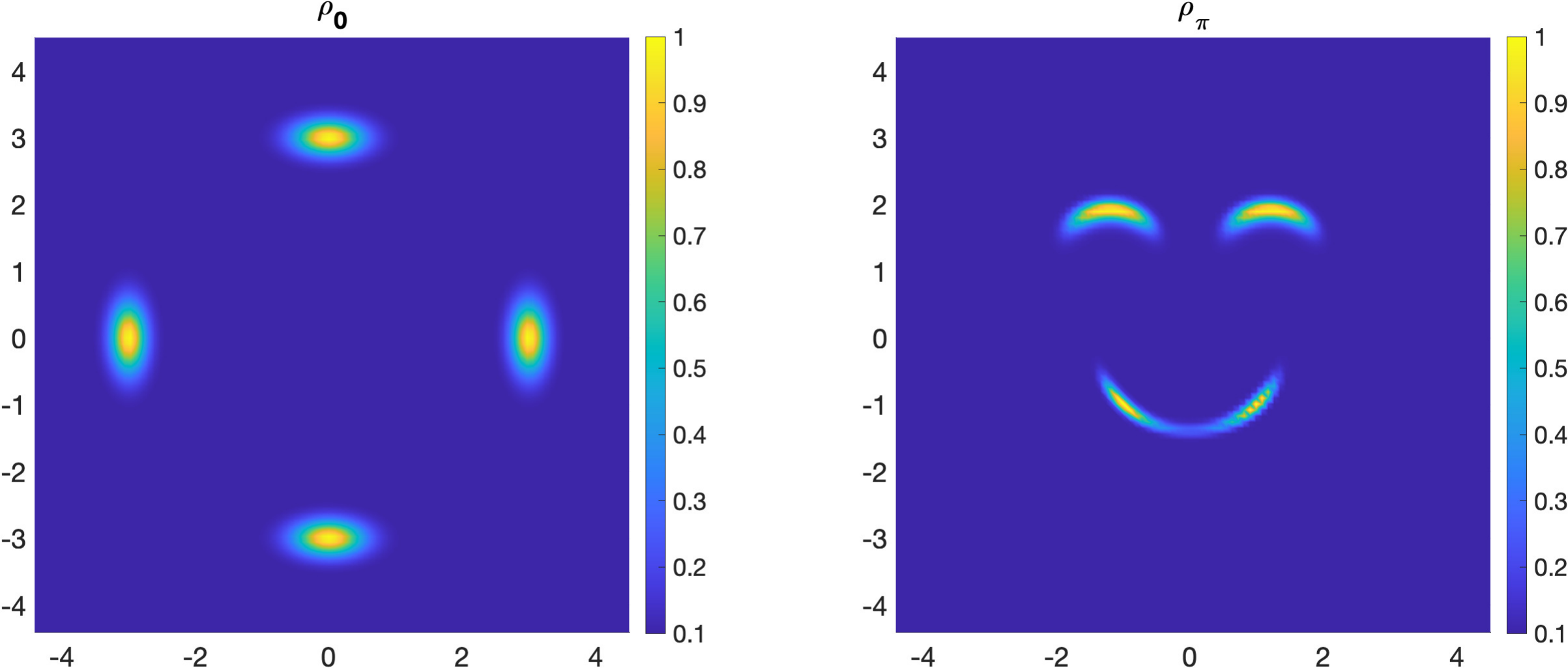} 
 \caption{The normalized initial density (left) and the target density: smiling triple banana (right). }\label{fig_sample1}
 \end{center}
 \end{figure} 
Set the computational parameters as diffusion constant $D=0.5$, time step $\Delta t = 0.01$, uniform grid size $\Delta x =\Delta y=0.09 $. Then with the amplitude of the mixture velocity $A=0.1$ and the frequency $k=8$, the time evolution of dynamic density $\rho_t$ is shown at time iteration $n_t=50, 200, 1200, 10000$ in Fig.\ref{fig_sample_time}. Moreover, the relative  root mean square error between the dynamic solution $\rho_t$ and the target density $\pi$ are shown  w.r.t time iterations using semilog plot in Fig.\ref{fig_sample_e} with different amplitude of the mixture velocity $A=0$ and $0.1$. Notice if $A=0$, then the Fokker-Planck equation \eqref{FPr} and the corresponding upwind scheme \eqref{F_num} are reduced to the reversible case, i.e., we only have the gradient flow part in \eqref{FPr}. The corresponding  $Q$-matrix for the reversible case ($\vec{u}=0$) of course also leads to an efficient sampling for the target density. However,   the additional incompressible convection with  larger mixture velocity $\vec{u}$, although makes the dynamics irreversible, can speed up the convergence to the target density function. This observation is shown in the semilog plot in Fig.\ref{fig_sample_e} for the decay of the relative  root mean square error with amplitude  $A=0.1$, compared with $A=0$.
The enhanced convergence and diffusion 	by a mixture velocity filed is a classical topic in the fluid dynamics and PDE analysis \cite{fannjiang1994convection, constantin2008diffusion}. This technique has a promising applications in sampling, Markov chain Monte Carlo and non-convex optimizations. We will leave this line of research as a future study.

 \begin{figure}
\includegraphics[scale=0.57]{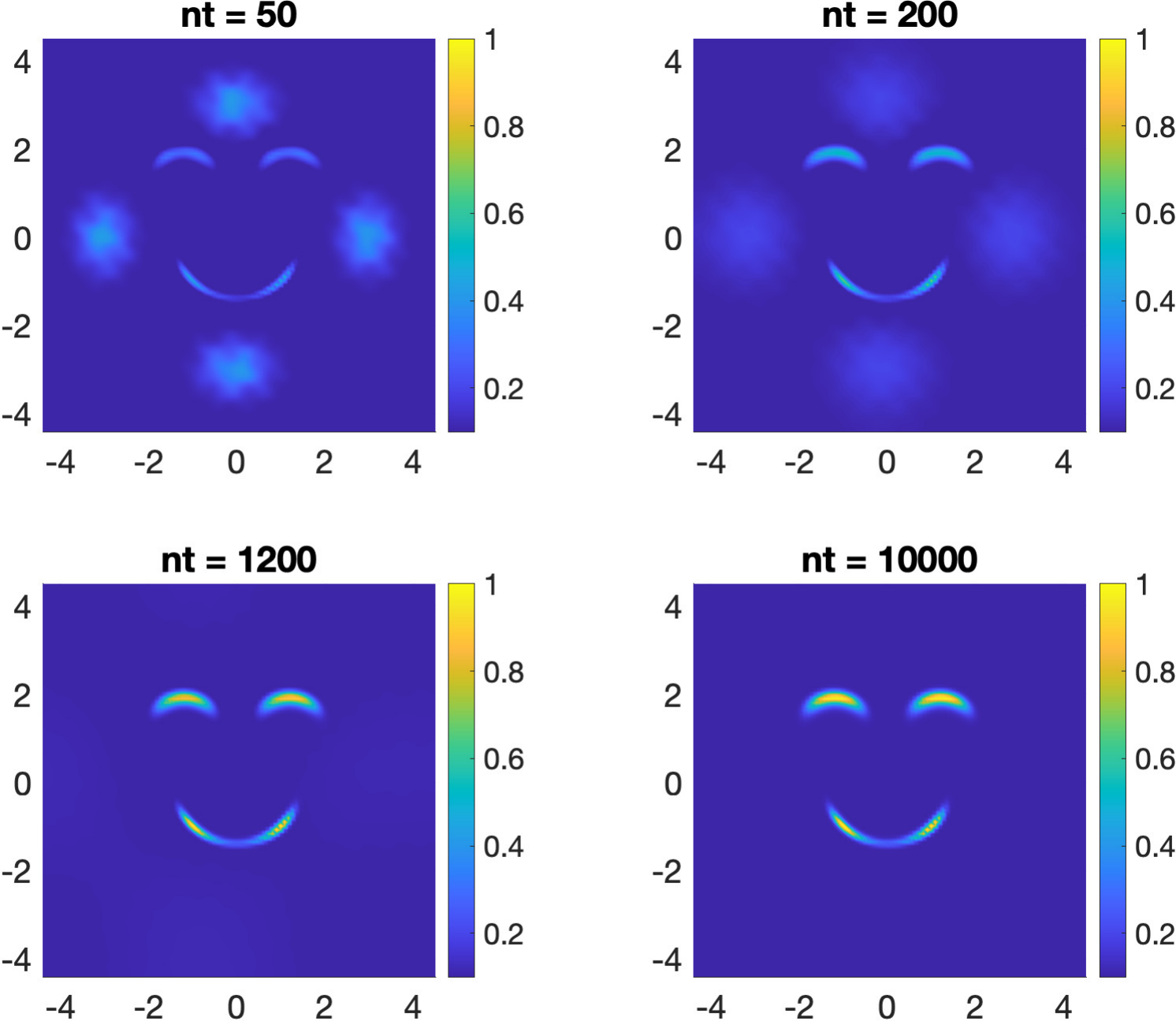} 
  \caption{The  time evolution of dynamic solution for sampling the smiling triple banana using  \eqref{F_num}. The amplitude of the mixture velocity is $A=0.1$ and snapshots are shown  at iterations $n_t=50, 200, 1200, 10000$. }\label{fig_sample_time}
   \end{figure}
 \begin{figure}
 \begin{center}
\includegraphics[scale=0.27]{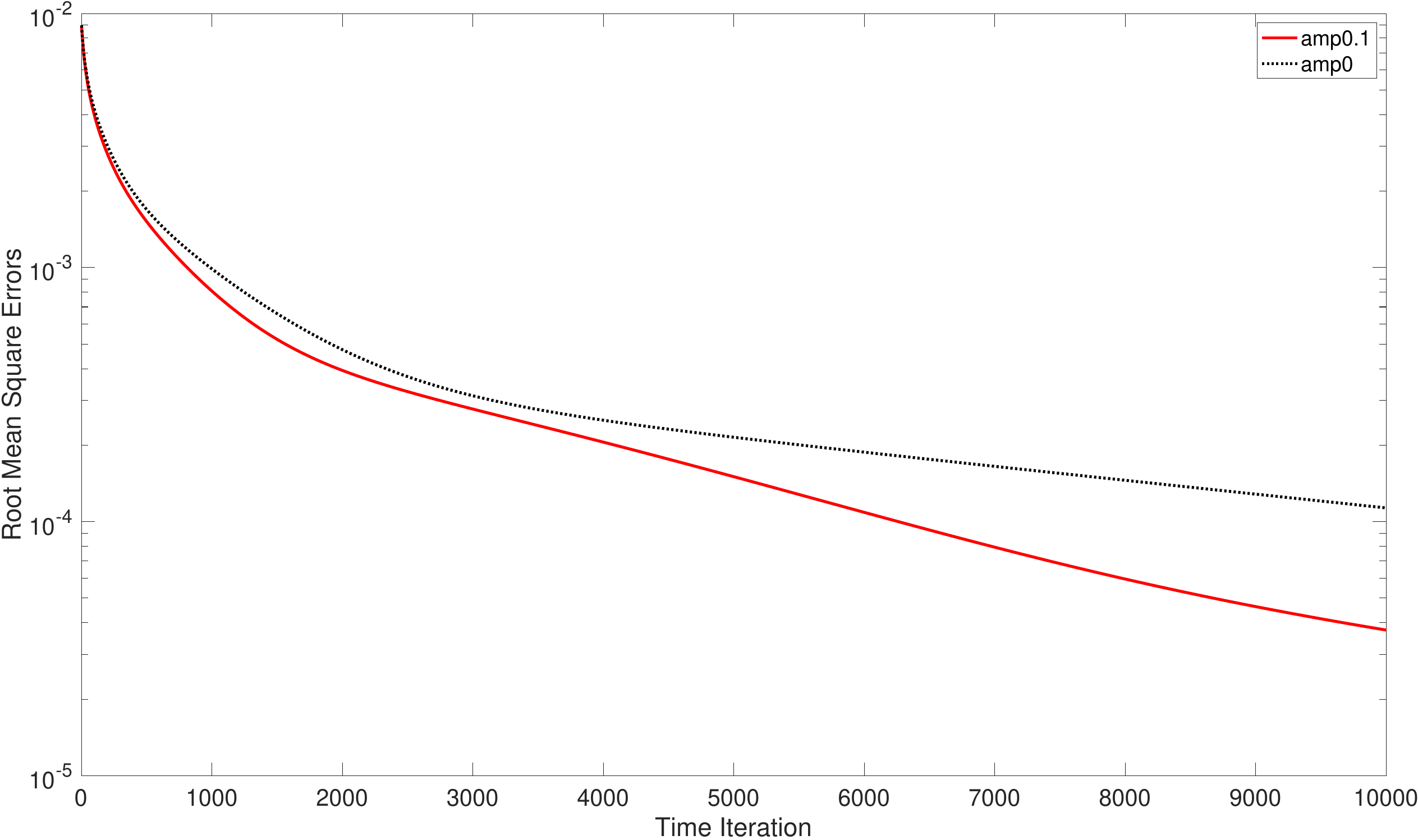} 
\caption{The relative  root mean square error between the dynamic density and the target triple banana density is shown in semilog plot w.r.t  time iterations. With different amplitude $A=0$ and $0.1$, we observe the enhancement of convergence brought by the incompressible mixture velocity $\vec{u}$.}\label{fig_sample_e} 
 \end{center} 
 \end{figure}

\subsubsection{Example: image transformation immersed in an incompressible flow}
In this example, we use  van Gogh's `The Starry Night' to simulate an image immersed in an irreversible dynamics but still converge to a given target image.
Choose an initial image with the same village view but with a purely blue sky, as shown in Fig.\ref{fig_VG1}(left). The target image `The Starry Night' is shown in Fig.\ref{fig_VG1}(right).
Each matrix exacted from these two images contains values of a color mode (R or G or B) and are both $N=256$ pixels in width and $M=203$
pixels in height.
\begin{figure}
\begin{center}
\includegraphics[scale=0.35]{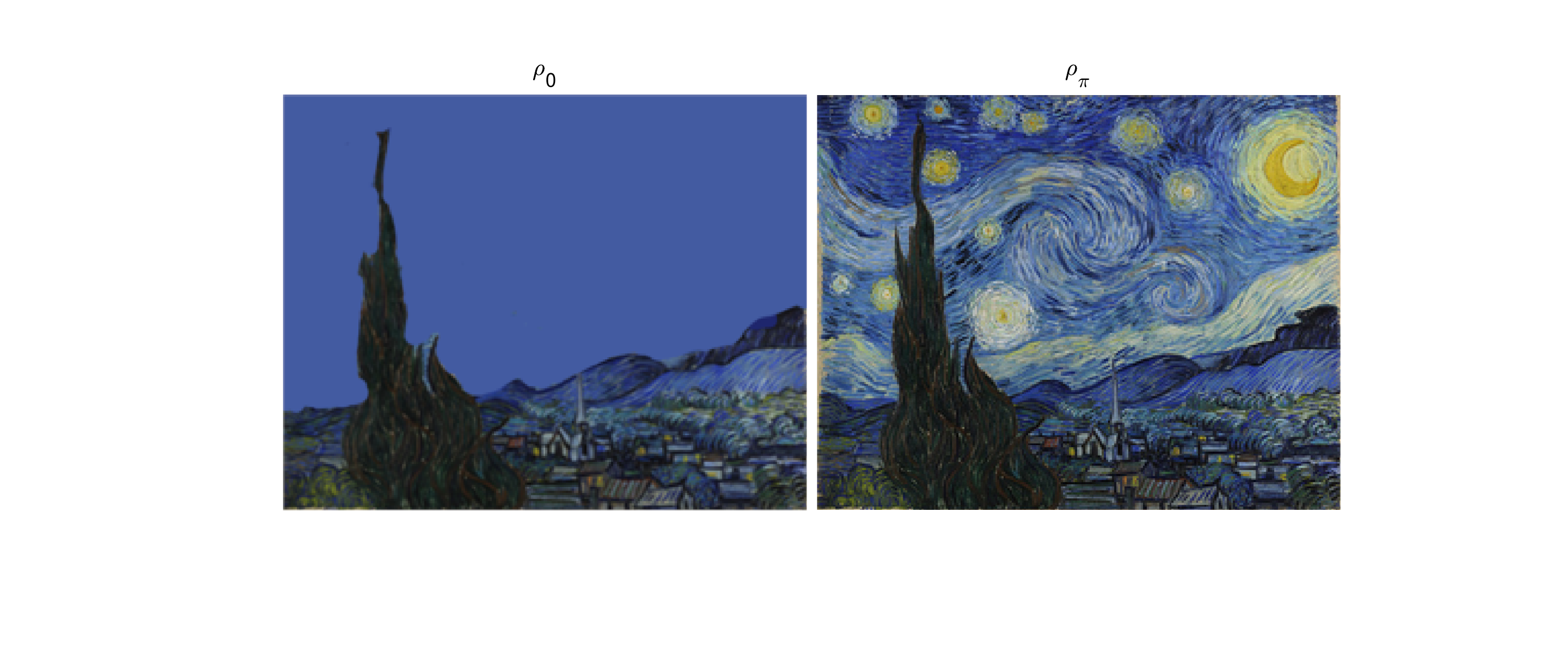} 
\caption{The  initial image (left) and the target image: The Starry Night(right).}\label{fig_VG1}
\end{center}
\end{figure}

Set the computational parameters as diffusion constant $D=0.4$, time step $\Delta t = 0.01$, uniform grid size $\Delta x = \frac{\pi}{N}, \, \Delta y=\frac{\pi}{M}$. Then with the amplitude of the mixture velocity $A=1000$ and the normalized wave number $k=8$, the time evolution of dynamic density $\rho_t$ is shown at time iteration $n_t=5, 80, 400, 2000$ in Fig.\ref{fig_VG_time}. Starting from a purely blue sky, we can see the night sky immersed in the incompressible flow, which is discretized using \eqref{streamU} with \eqref{psi},  becomes starry, unbalanced and distorted. At $n_t=2000$, the image is very close to the target image while the more unbalanced inbetweening image is shown at $n_t=400$.
Moreover, the relative  root mean square error between the dynamic solution $\rho_t$ and the target density $\pi$ are shown  w.r.t time iterations using semilog plot in Fig.\ref{fig_VG_e}.  We still observe that larger amplitude of the mixture velocity has an enhancement for the convergence of dynamic solution to its steady state.

\begin{figure}
\begin{center}
\includegraphics[scale=0.53]{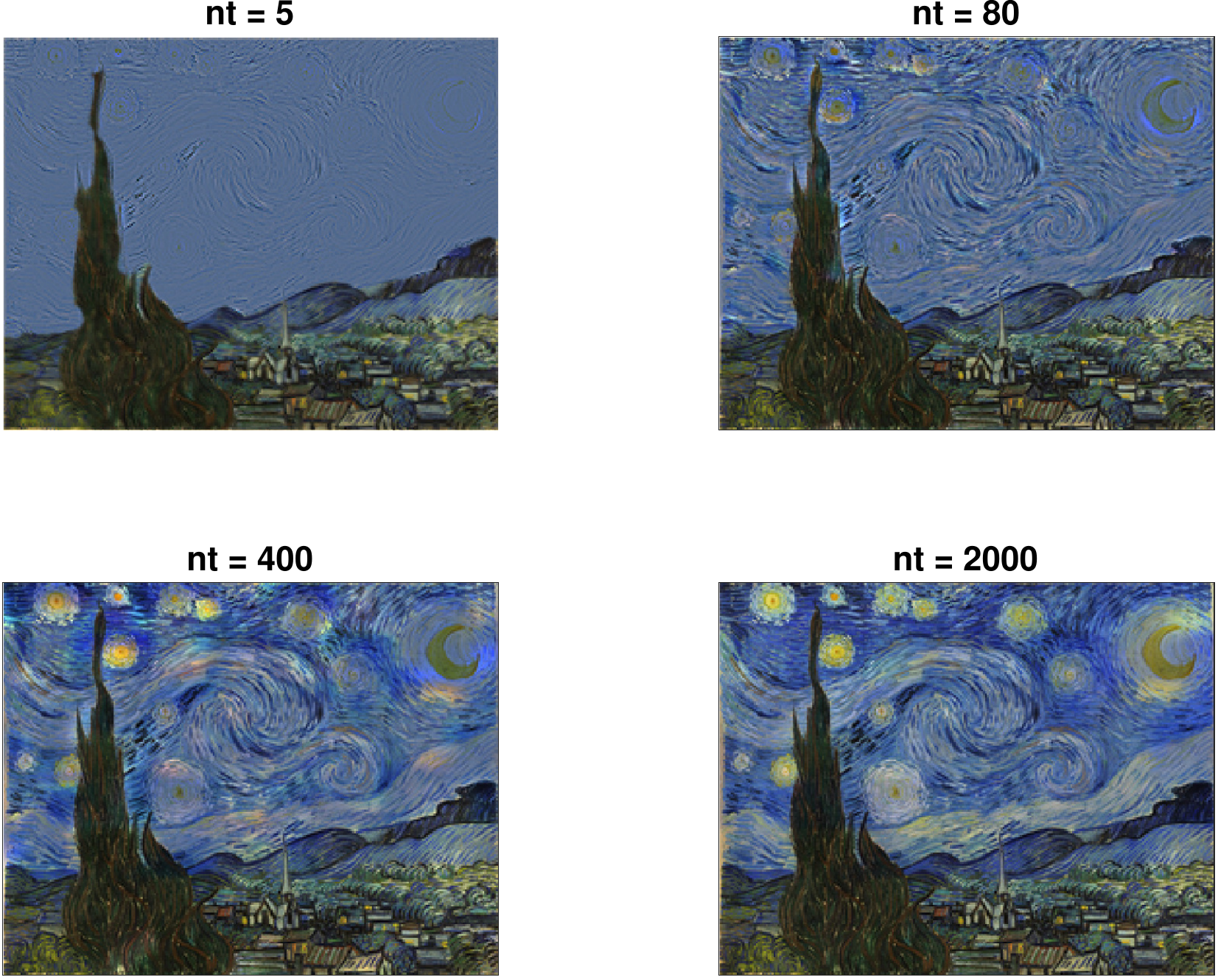} 
\caption{The  time evolution of image transformations immersed in an incompressible flow using scheme  \eqref{F_num}. The amplitude of the mixture velocity is $A=1000$ and snapshots are shown  at iterations $n_t=5, 80, 400, 2000$.} \label{fig_VG_time}
\end{center}
\end{figure}

\begin{figure}
\begin{center}
\includegraphics[scale=0.27]{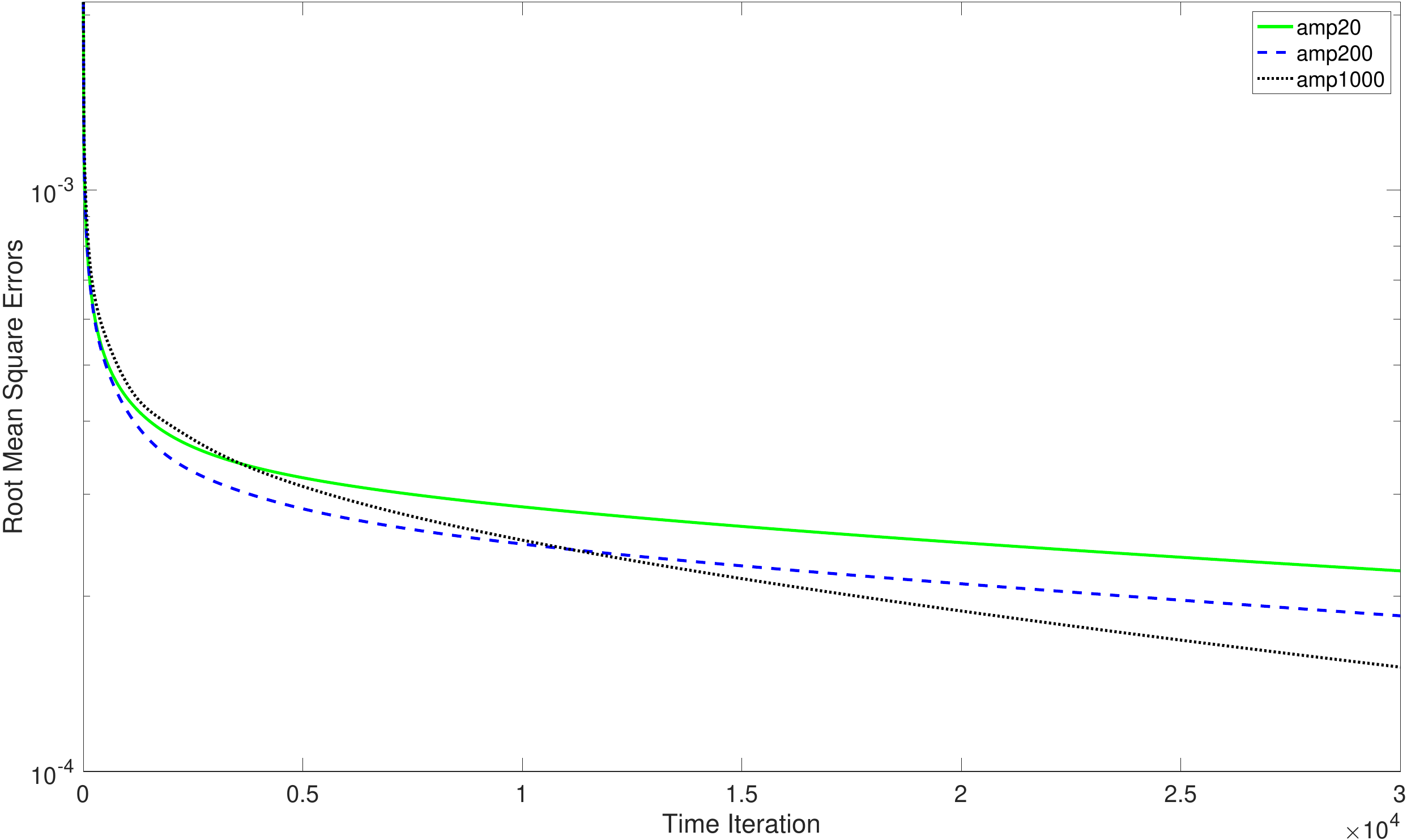} 
\caption{The relative  root mean square error between the dynamic image  in blue color-modes and the target image is shown in semilog plot w.r.t  time iterations. Enhanced convergence is shown with different amplitude $A=20,\, 200,\, 1000$.}\label{fig_VG_e}
\end{center}
\end{figure}

\subsection{Simulations for general irreversible dynamics without information of invariant measure}\label{sec_5.2van}
In this section, assume in  irreversible drift-diffusion process \eqref{sde-y},  we only know  the  drift vector field $\vec{b}$. Then we use the upwind scheme \eqref{mp} in 2D structured grids to simulate a  stochastic  Van der Pol oscillator.

\subsubsection{Example: Stochastic  Van der Pol oscillator  model} The famous Van der Pol model is first proposed by Van der Pol to describe electrical circuit and also has numerous extended applications in biology, pharmacology  and seismology. For instance the  FitzHugh-Nagumo model describing the  excitation and propagation  of sodium and potassium ions in a neuron \cite{Izhikevich:2006}. There are lots of classical investigations on the limit circle, bifurcations and chaos on this model; c.f. \cite{varigonda2001dynamics}.

To illustrate our numerical scheme for the irreversible dynamics without steady state information,     consider a stochastic version of  Van der Pol oscillator  model
\begin{equation}
\begin{aligned}
{\ud x} =\alpha(  x - \frac{x^3}{3}+y) \ud t+ \sqrt{2 \eps} \ud B;\\
{\ud y} =(\delta - x) \ud t+ \sqrt{2 \eps} \ud B
\end{aligned}
\end{equation}
with a set of classical parameters $\alpha=10$ and $\delta =0$ or $\delta=1$.  After adding the Brownian motion with $\eps>0$, this is a typical example which can be decomposed as a gradient flow part and a Hamiltonian part.

The corresponding Fokker-Planck equation is
\begin{equation}\label{FP_van}
\begin{aligned}
\pt_t \rho =\nabla \cdot \bbs{\eps \nabla \rho - \vec{b} \rho}; \qquad \vec{b} = \big(\alpha(  x - \frac{x^3}{3}+y), \,   (\delta - x)\big).
\end{aligned}
\end{equation}
We take domain as $\Omega=[-3,4]\times[-3,3]$, noise level as $\eps=0.1$, computational parameters as $\Delta t=0.05$, $\Delta x=0.07, \Delta y=0.06$ and take the no-flux boundary condition \eqref{bc}. 
Then based on \eqref{num355} (i.e., \eqref{Van_scheme}), starting from the  initial density 
$
\rho_0 =0.2,
$
we compute the time evolution of the density function $\rho_t$. 

With parameter $\delta=0$, the numerical solution at time iteration $n_t = 1000, 5000, 40000$ are shown in Fig \ref{fig_Lcircle}. Meanwhile, the numerical steady state $\pi^\8$ is computed by setting the time iteration as $n_t = 200000$, because in Proposition \ref{prop_pi_con} we have proved the exponential ergodicity for the unconditionally explicit scheme \eqref{timeG}. The root mean square error between the numerical solution $\rho_t$ and the numerical steady state $\pi^\8$ is shown in the semilog plot in Fig \ref{fig_Lcircle}(downright) in terms of time iterations. We can see for $\delta=0$, the steady state density tends to concentrate near a large limit circle in Fig \ref{fig_Lcircle} (downleft), which is the stable limit circle for the deterministic  Van der Pol oscillator in the large damping regime (a.k.a. relaxation oscillation).
\begin{figure}
\begin{center}
\includegraphics[scale=0.52]{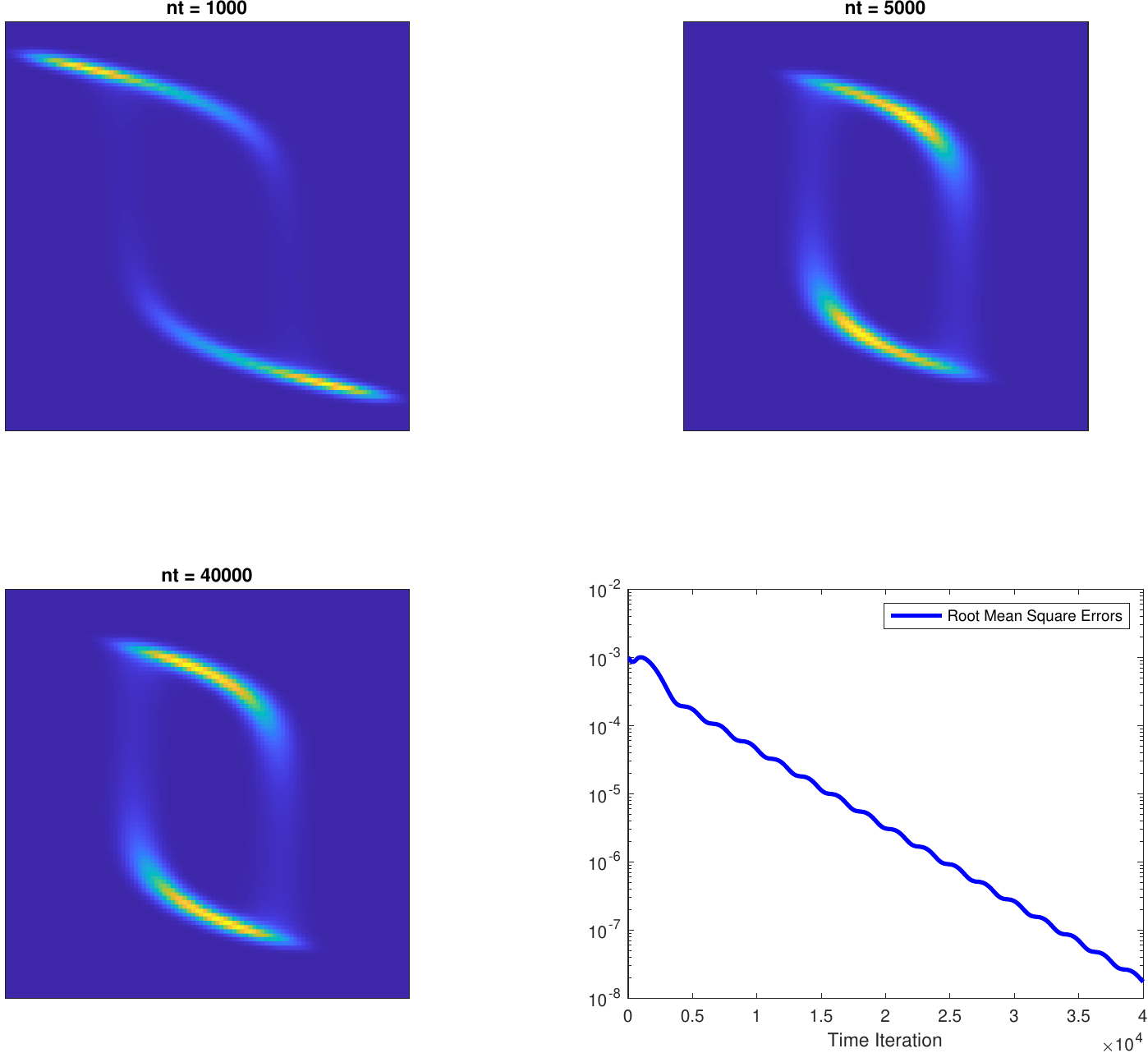} 
\caption{The time evolution of density function to \eqref{FP_van} with $\delta=0$. At time iteration $n_t=40000$, the steady state tends to concentrate near the larger limit circle (downleft). The root mean square error between the numerical solution $\rho_t$ and the numerical steady state $\pi^\8$ is shown in the semilog plot (downright).   }\label{fig_Lcircle}
\end{center}
\end{figure}

As a comparison, with different parameter $\delta=1$, the numerical solution at time iteration $n_t = 1000, 5000, 40000$ are shown in Fig \ref{fig_Scircle}. The root mean square error between the numerical solution $\rho$ and the numerical steady state $\pi^\8$ is shown in the semilog plot in Fig \ref{fig_Scircle}(downright), where the numerical steady state is still obtained by setting the time iteration as $n_t = 200000$. We can see clearly the steady state density now  concentrates near a smaller region in Fig \ref{fig_Scircle}(downleft). Indeed, this small region is near another smaller limit circle for the deterministic  Van der Pol oscillator, but due to the random noise, the  steady state density does not exactly concentrate only on the smaller limit circle.
\begin{figure}
\begin{center}
\includegraphics[scale=0.52]{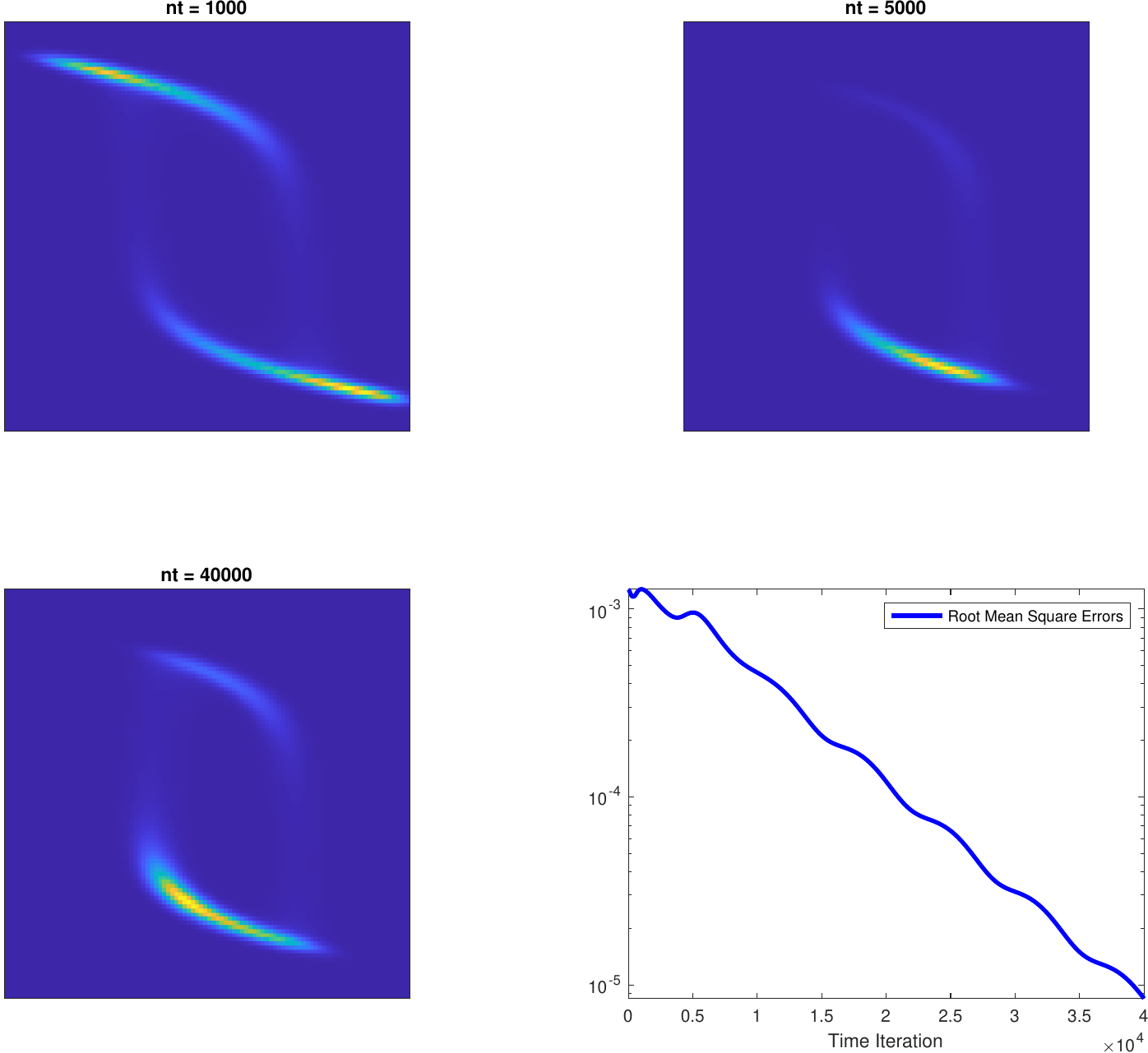} 
\caption{The time evolution of density function to \eqref{FP_van} with $\delta=1$. The steady state tends to concentrate near the smaller limit circle (downleft). In terms of time iterations, the root mean square error  between the numerical solution $\rho$ and the numerical steady state $\pi$ is shown in the semilog plot (downright). }\label{fig_Scircle}
\end{center}
\end{figure}

\section*{Acknowledgements}
Jian-Guo Liu was supported in part by NSF under awards DMS-2106988 and by NSF RTG grant DMS-2038056. Yuan Gao was supported by NSF under awards DMS-2204288.

\bibliographystyle{plain}
\bibliography{tst_ng}

\appendix

\section{Finite volume scheme for 2D structured grids}

\subsection{$\pi$-symmetric upwind scheme in 2D  with a given invariant measure}\label{app1}

After choosing $\vec{u}$ satisfying \eqref{streamU}, we use \eqref{F_num} to present the $\pi$-symmetric upwind scheme in a structured 2D domain $\Omega:=[a,b]\times[c,d].$ 

\begin{enumerate}[(i)]
\item Define the rectangle cells as
\begin{equation}\label{A1}
C_{ij} = ((i-1)\Delta x, i \Delta x) \times ((j-1)\Delta y, j\Delta y), \quad i=1, \cdots, N, \,\, j=1, \cdots, M.
\end{equation}
Denote the approximated density on $C_{ij}$ as $\rho_{i,j}$; Denote the  cell centers $(x_{i},y_{j})$ as
\begin{equation}\label{A2}
x_{i} = a+ (i-\frac12)\Delta x, \quad y_{j}= c+ (j-\frac12) \Delta y,  \quad i=1, \cdots, N, \,\, j=1, \cdots, M;
\end{equation}
\item 
Denote the discrete drift $(\vec{u}\cdot \vec{n})_{i,j}|\Gamma_{i.j}|$ of cell $C_{ij}$ on each (right/left/up/down) faces $\Gamma_{ij}$ along outer normal $\vec{n}$  as
\begin{align*}
 (\vec{u}\cdot \vec{n})^R_{i,j} \Delta y = \psi(x_i+\frac{\Delta x}{2}, y_j+ \frac{\Delta y}{2}) - \psi(x_i+\frac{\Delta x}{2}, y_j- \frac{\Delta y}{2}),\\
 (\vec{u}\cdot \vec{n})^L_{i,j} \Delta y = \psi(x_i-\frac{\Delta x}{2}, y_j- \frac{\Delta y}{2}) - \psi(x_i-\frac{\Delta x}{2}, y_j+ \frac{\Delta y}{2}),\\
 (\vec{u}\cdot \vec{n})^U_{i,j} \Delta x = \psi(x_i-\frac{\Delta x}{2}, y_j+ \frac{\Delta y}{2}) - \psi(x_i+\frac{\Delta x}{2}, y_j+ \frac{\Delta y}{2}),\\
       (\vec{u}\cdot \vec{n})^D_{i,j} \Delta x = \psi(x_i+\frac{\Delta x}{2}, y_j- \frac{\Delta y}{2}) - \psi(x_i-\frac{\Delta x}{2}, y_j- \frac{\Delta y}{2}).
\end{align*}
\item Based on \eqref{F_num}, using the negative part of the discrete drift above, define the  inward flux $F$  into cell $C_{ij}$ from four cell faces (right/left/up/down) $\Gamma_{ij}$ as
\begin{equation}
\begin{aligned}
&F^{\text{R}}_{i, j} = \Delta y \bbs{ \frac{D(\pi_{i+1,j}+\pi_{i,j})}{2\Delta x } +(\vec{u}\cdot \vec{n})^{R-}_{i,j} }\bbs{ \frac{\rho_{i+1, j}}{\pi_{i+1,j}} - \frac{\rho_{i,j}}{\pi_{i,j}}} ,\\
&F^{\text{L}}_{i, j} = \Delta y \bbs{ \frac{D(\pi_{i-1,j}+\pi_{i,j})}{2\Delta x }+ (\vec{u}\cdot \vec{n})^{L-}_{i,j} }\bbs{ \frac{\rho_{i-1, j}}{\pi_{i-1,j}} - \frac{\rho_{i,j}}{\pi_{i,j}}} ,\,\, i=2, \cdots, N-1, \,\, j=1, \cdots, M;\\
&F^{\text{U}}_{i, j} = \Delta x \bbs{ \frac{D(\pi_{i,j+1}+\pi_{i,j})}{2\Delta y} + (\vec{u}\cdot \vec{n})^{U-}_{i,j}}\bbs{ \frac{\rho_{i, j+1}}{\pi_{i,j+1}} - \frac{\rho_{i,j}}{\pi_{i,j}}},\\
&F^{\text{D}}_{i, j} = \Delta x  \bbs{ \frac{D(\pi_{i,j-1}+\pi_{i,j})}{2\Delta y} +(\vec{u}\cdot \vec{n})^{D-}_{i,j}} \bbs{ \frac{\rho_{i, j-1}}{\pi_{i,j+1}} - \frac{\rho_{i,j}}{\pi_{i,j}}},\,\,i=1, \cdots, N, \,\, j=2, \cdots, M-1.
\end{aligned}
\end{equation} 
\item Impose the no-flux boundary condition
\begin{equation}\label{bc_p}
\begin{aligned}
F^{\text{L}}_{1, j} = F^{\text{R}}_{N, j}=0, \quad j=1, \cdots, M; \quad
F^{\text{D}}_{i,1} = F^{\text{U}}_{i, M} =0,\quad  i=1, \cdots, N.
\end{aligned}
\end{equation}
\end{enumerate}

Then the continuous-time finite volume scheme
is
\begin{equation}\label{fp2d_p}
\begin{aligned}
\dot{\rho}_{i,j}\Delta x \Delta y = F^{\text{R}}_{i, j}+ F^{\text{L}}_{i, j}+ F^{\text{U}}_{i, j} + F^{\text{D}}_{i, j}
\end{aligned}
\end{equation}
for $i=1, \cdots, N, \, j=1, \cdots, M$ with the no-flux boundary condition \eqref{bc_p}. Here $\dot{\rho}_{i,j}$ refers to the time derivative of $\rho_{i,j}$. 
Denote $g_{i,j}:=\frac{\rho_{i,j}}{\pi_{i,j}}$. Recast \eqref{fp2d_p} as a five-point scheme
\begin{align*}
&\pi_{i,j} \dot{g}_{i,j} = \frac{1}{\Delta x^2}\bbs{\frac{D(\pi_{i+1,j}+\pi_{i,j})}{2} +  \Delta x(\vec{u}\cdot \vec{n})^{R-}_{i,j}  } g_{i+1, j} +  \frac{1}{\Delta x^2}\bbs{\frac{D(\pi_{i-1,j}+\pi_{i,j})}{2} +  \Delta x (\vec{u}\cdot \vec{n})^{L-}_{i,j}} g_{i-1, j}  \\
&+  \frac{1}{\Delta y^2}\bbs{\frac{D(\pi_{i,j+1}+\pi_{i,j})}{2 } +  \Delta y (\vec{u}\cdot \vec{n})^{U-}_{i,j}} g_{i, j+1} + \frac{1}{\Delta y^2}\bbs{\frac{D(\pi_{i,j-1}+\pi_{i,j})}{2 } +  \Delta y (\vec{u}\cdot \vec{n})^{D-}_{i,j} } g_{i, j-1} - \lambda_{i,j} g_{i,j},\\
& \lambda_{i,j} := \frac{D(\pi_{i+1,j}+2\pi_{i,j} +\pi_{i-1,j})}{2\Delta x^2}+ \frac{D(\pi_{i,j+1}+2\pi_{i,j} +\pi_{i,j-1})}{2\Delta y^2} + \frac{(\vec{u}\cdot \vec{n})^{R-}_{i,j} + (\vec{u}\cdot \vec{n})^{L-}_{i,j}  }{\Delta x} + \frac{(\vec{u}\cdot \vec{n})^{U-}_{i,j} + (\vec{u}\cdot \vec{n})^{D-}_{i,j}   }{\Delta y}.
\end{align*}

Following the idea for the unconditionally stable explicit scheme \eqref{num355}, choose constant 
$\displaystyle a> \max_{i,j} \frac{\lambda_{i,j}}{\pi_{i,j}}.$
 Using constant $a>0$, we can construct   an  unconditionally stable explicit scheme, which can be regarded as a new Markov chain   with a transition probability $\tilde{K}^*$.
With $Q$ defined in \eqref{Q-p}, the time discretization is
\begin{equation}
\frac{\rho^{k+1}-\rho^k}{\Delta t} = Q^* \rho^k + a(\rho^k-\rho^{k+1}).
\end{equation}  
Thus the transition probability $\tilde{K}^*$ is given by
\begin{equation}\label{time-P}
\rho^{k+1} = \tilde{K}^* \rho^k, \quad \tilde{K}^*:=I+ \frac{\Delta t Q^*}{1+a \Delta t},\,\, \sum_j \tilde{K}_{ij} = 1,\,\, \tilde{K}_{ij} \geq 0.
\end{equation}
Notice we have chosen $\vec{u}$ satisfying \eqref{div0} and thus the equivalent flux $F_{ji}$ in \eqref{F_num} still defines a stochastic $Q$-matrix.
 With the notation $g_{i,j}=\frac{\rho_{i,j}}{\pi_{i,j}}$,  \eqref{time-P} reads as
 three parts: interiors, four sides, and four corners. 
For  interior cells $i=2, \cdots, N-1, \ j=2, \cdots, M-1$:
\begin{equation}\label{num2d_pi}
\begin{aligned}
 \rho^{k+1}_{i,j}& = \frac{\Delta t}{(1+a \Delta t)\Delta x^2}\bbs{\frac{D(\pi_{i+1,j}+\pi_{i,j})}{2}  +  \Delta x (\vec{u}\cdot \vec{n})^{R-}_{i,j}} g^k_{i+1, j} \\
 &+  \frac{\Delta t}{(1+a \Delta t)\Delta x^2}\bbs{\frac{D(\pi_{i-1,j}+\pi_{i,j})}{2} +  \Delta x(\vec{u}\cdot \vec{n})^{L-}_{i,j} } g^k_{i-1, j}   \\
&+  \frac{\Delta t}{(1+a \Delta t)\Delta y^2}\bbs{\frac{D(\pi_{i,j+1}+\pi_{i,j})}{2 } +  \Delta y (\vec{u}\cdot \vec{n})^{U-}_{i,j}} g^k_{i, j+1}\\
&+ \frac{\Delta t}{(1+a \Delta t)\Delta y^2}\bbs{\frac{D(\pi_{i,j-1}+\pi_{i,j})}{2 } +  \Delta y (\vec{u}\cdot \vec{n})^{D-}_{i,j} } g^k_{i, j-1} + \frac{ \pi_{i,j} + \Delta t(a \pi_{i,j}-\lambda_{i,j})}{1+a \Delta t} g^k_{i,j}. 
\end{aligned}
\end{equation}

For left side cells $i=1, \ j=2, \cdots, M-1$:
\begin{equation}
\begin{aligned}
 &\rho^{k+1}_{1,j} = \frac{\Delta t}{(1+a \Delta t)\Delta x^2}\bbs{\frac{D(\pi_{2,j}+\pi_{1,j})}{2} +  \Delta x (\vec{u}\cdot \vec{n})^{R-}_{1,j}} g^k_{2, j} \\
& \quad +  \frac{\Delta t}{(1+a \Delta t)\Delta y^2}\bbs{\frac{D(\pi_{1,j+1}+\pi_{1,j})}{2}+  \Delta y (\vec{u}\cdot \vec{n})^{U-}_{1,j} } g^k_{1, j+1} \\
&\quad + \frac{\Delta t}{(1+a \Delta t)\Delta y^2}\bbs{D\frac{(\pi_{1,j-1}+\pi_{1,j})}{2} +  \Delta y (\vec{u}\cdot \vec{n})^{D-}_{1,j} } g^k_{1, j-1} 
+ \frac{ \pi_{1,j} + \Delta t(a \pi_{1,j}-\lambda_{1,j})}{1+a \Delta t} g^k_{1,j},\\
& \lambda_{1,j} :=\frac{D(\pi_{2,j}+\pi_{1,j} )}{2\Delta x^2}+ \frac{D(\pi_{1,j+1}+2\pi_{1,j} +\pi_{1,j-1})}{2\Delta y^2} + \frac{(\vec{u}\cdot \vec{n})^{R-}_{1,j}  }{\Delta x} + \frac{(\vec{u}\cdot \vec{n})^{U-}_{1,j} + (\vec{u}\cdot \vec{n})^{D-}_{1,j} }{\Delta y}.
\end{aligned}
\end{equation}
Similar for the right, down and up side cells.

For left-down corner cell $i=1, \ j=1$:
\begin{equation}
\begin{aligned}
 &\rho^{k+1}_{1,1} = \frac{\Delta t}{(1+a \Delta t)\Delta x^2}\bbs{\frac{D(\pi_{2,1}+\pi_{1,1} )}{2} +  \Delta x  (\vec{u}\cdot \vec{n})^{R-}_{1,1}  } g^k_{2, 1} \\
 & +
   \frac{\Delta t}{(1+a \Delta t)\Delta y^2}\bbs{\frac{D(\pi_{1,2}+\pi_{1,1} )}{2} +  \Delta y(\vec{u}\cdot \vec{n})^{U-}_{1,1} } g^k_{1, 2} + \frac{ \pi_{1,1} + \Delta t(a \pi_{1,1}-\lambda_{1,1})}{1+a \Delta t} g^k_{1,1} \\
& \quad  \lambda_{1,1} :=\frac{D(\pi_{2,1}+\pi_{1,1} )}{2\Delta x^2}+ \frac{D(\pi_{1,2}+\pi_{1,1} )}{2\Delta y^2} + \frac{(\vec{u}\cdot \vec{n})^{R-}_{1,1} }{\Delta x} + \frac{(\vec{u}\cdot \vec{n})^{U-}_{1,1}}{\Delta y}.
\end{aligned}
\end{equation}
Similar for the left-up, right-down and right-up corner cells.

\subsection{Upwind scheme in 2D for general irreversible process}\label{app2}
Take a 2D domain as $\Omega:= [a,b]\times [c,d].$ 
We present the upwind scheme \eqref{mp} based on structured grids for a Fokker-Planck equation on the 2D domain $\Omega$. Let the grid size be $\Delta x= \frac{b-a}{N}, \, \Delta y = \frac{d-c}{M}$. 
\begin{enumerate}[(i)]
\item Define the rectangle cells and cell centers as \eqref{A1}, and \eqref{A2}.
\item Denote the drift 
 $
\vec{b}=: (u,v)
$ of cell $C_{ij}$ on the bisection point on edges as
\begin{equation}
u_{i\pm\frac12, j} = u(x_{i} \pm \frac12 \Delta x, y_j), \quad  v_{i, j\pm\frac12} = v(x_i, y_{j} \pm \frac12 \Delta y);
\end{equation}
\item Based on \eqref{mp}, define the  inward flux $F$  into cell $C_{ij}$ from four cell faces as
\begin{equation}
\begin{aligned}
&F^{\text{R}}_{i, j} = \Delta y \bbs{ \frac{D(\rho_{i+1, j} - \rho_{i,j})}{\Delta x} + u^-_{i+\frac12, j} \rho_{i+1,j}  - u^+_{i+\frac12, j} \rho_{i,j}},\\
&F^{\text{L}}_{i, j} = \Delta y \bbs{ \frac{D(\rho_{i-1, j} - \rho_{i,j})}{\Delta x} + u^+_{i-\frac12, j} \rho_{i-1,j}  - u^-_{i-\frac12, j} \rho_{i,j}},\,\,  i=2, \cdots, N-1, \,\, j=1, \cdots, M;\\
&F^{\text{U}}_{i, j} = \Delta x \bbs{ \frac{D(\rho_{i, j+1} - \rho_{i,j})}{\Delta y} + v^-_{i, j+\frac12} \rho_{i,j+1}  - v^+_{i, j+\frac12} \rho_{i,j}},\\
&F^{\text{D}}_{i, j} = \Delta x \bbs{ \frac{D(\rho_{i, j-1} - \rho_{i,j})}{\Delta y} + v^+_{i,j-\frac12j} \rho_{i,j-1}  - v^-_{i, j-\frac12} \rho_{i,j}},\,\,i=1, \cdots, N, \,\, j=2, \cdots, M-1.
\end{aligned}
\end{equation} 
\item Impose the no-flux boundary condition
\begin{equation}\label{bc}
\begin{aligned}
F^{\text{L}}_{1, j} = F^{\text{R}}_{N, j}=0, \quad j=1, \cdots, M; \quad 
F^{\text{D}}_{i,1} = F^{\text{U}}_{i, M} =0,\quad  i=1, \cdots, N.
\end{aligned}
\end{equation}
\end{enumerate}

Then the continuous-time finite volume scheme
is
\begin{equation}\label{fp2d}
\begin{aligned}
\dot{\rho}_{i,j}\Delta x \Delta y = F^{\text{R}}_{i, j}+ F^{\text{L}}_{i, j}+ F^{\text{U}}_{i, j} + F^{\text{D}}_{i, j}
\end{aligned}
\end{equation}
for $i=1, \cdots, N, \, j=1, \cdots, M$ with the no-flux boundary condition \eqref{bc}. Here $\dot{\rho}_{i,j}$ refers to the time derivative of $\rho_{i,j}$. 
Recast \eqref{fp2d} as a five-point scheme
\begin{align*}
\dot{\rho}_{i,j} = \frac{1}{\Delta x^2}&\bbs{D +  \Delta x u^-_{i+\frac12, j} } \rho_{i+1, j} +  \frac{1}{\Delta x^2}\bbs{D +  \Delta x u^+_{i-\frac12, j} } \rho_{i-1, j}  \\
&+  \frac{1}{\Delta y^2}\bbs{D +  \Delta y v^-_{i, j+\frac12} } \rho_{i, j+1} + \frac{1}{\Delta y^2}\bbs{D +  \Delta y v^+_{i, j-\frac12} } \rho_{i, j-1} - \lambda_{i,j} \rho_{i,j},\\
& \lambda_{i,j} := \frac{2D}{\Delta x^2}+ \frac{2D}{\Delta y^2} + \frac{u^+_{i+\frac12, j} +u^-_{i-\frac12, j} }{\Delta x} + \frac{v^+_{i, j+\frac12}+ v^-_{i, j-\frac12} }{\Delta y}.
\end{align*}

Following the unconditionally stable explicit scheme \eqref{num355}, define $a:= \max_{i,j} \lambda_{i,j}$. Then \eqref{num355} reads as
 three parts: interiors, four sides, and four corners.

For  interior cells $i=2, \cdots, N-1, \ j=2, \cdots, M-1$:
\begin{equation}\label{Van_scheme}
\begin{aligned}
 &\rho^{k+1}_{i,j} = \frac{\Delta t}{(1+a \Delta t)\Delta x^2}\bbs{D +  \Delta x u^-_{i+\frac12, j} } \rho^k_{i+1, j} +  \frac{\Delta t}{(1+a \Delta t)\Delta x^2}\bbs{D +  \Delta x u^+_{i-\frac12, j} } \rho^k_{i-1, j}  \\
&+  \frac{\Delta t}{(1+a \Delta t)\Delta y^2}\bbs{D +  \Delta y v^-_{i, j+\frac12} } \rho^k_{i, j+1} + \frac{\Delta t}{(1+a \Delta t)\Delta y^2}\bbs{D +  \Delta y v^+_{i, j-\frac12} } \rho^k_{i, j-1} + \frac{ 1 + \Delta t(a-\lambda_{i,j})}{1+a \Delta t} \rho^k_{i,j} \\
&\quad  \lambda_{i,j} := \frac{2D}{\Delta x^2}+ \frac{2D}{\Delta y^2} + \frac{u^+_{i+\frac12, j} +u^-_{i-\frac12, j} }{\Delta x} + \frac{v^+_{i, j+\frac12}+ v^-_{i, j-\frac12} }{\Delta y}.
\end{aligned}
\end{equation}

For left side cells $i=1, \ j=2, \cdots, M-1$:
\begin{equation}
\begin{aligned}
 &\rho^{k+1}_{1,j} = \frac{\Delta t}{(1+a \Delta t)\Delta x^2}\bbs{D +  \Delta x u^-_{\frac32, j} } \rho^k_{2, j} 
+  \frac{\Delta t}{(1+a \Delta t)\Delta y^2}\bbs{D +  \Delta y v^-_{1, j+\frac12} } \rho^k_{1, j+1} \\
&\quad + \frac{\Delta t}{(1+a \Delta t)\Delta y^2}\bbs{D +  \Delta y v^+_{1, j-\frac12} } \rho^k_{1, j-1} 
+ \frac{ 1 + \Delta t(a-\lambda_{i,j})}{1+a \Delta t} \rho^k_{i,j}\\
& \lambda_{1,j} := \frac{D}{\Delta x^2}+ \frac{2D}{\Delta y^2} 
+ \frac{u^+_{\frac32, j}  }{\Delta x} + \frac{v^+_{1, j+\frac12}+ v^-_{1, j-\frac12} }{\Delta y}.
\end{aligned}
\end{equation}
Similar for the right, down and up side cells.

For left-down corner cell $i=1, \ j=1$:
\begin{equation}
\begin{aligned}
 &\rho^{k+1}_{1,1} = \frac{\Delta t}{(1+a \Delta t)\Delta x^2}\bbs{D +  \Delta x u^-_{\frac32, 1} } \rho^k_{2, 1} +  \frac{\Delta t}{(1+a \Delta t)\Delta y^2}\bbs{D +  \Delta y v^-_{1, \frac32} } \rho^k_{1, 2} 
+ \frac{ 1 + \Delta t(a-\lambda_{1,1})}{1+a \Delta t} \rho^k_{i,j}\\
& \quad  \lambda_{1,1} := \frac{D}{\Delta x^2}+ \frac{D}{\Delta y^2} + \frac{u^+_{\frac32, 1} }{\Delta x} + \frac{v^+_{1,\frac32}}{\Delta y}.
\end{aligned}
\end{equation}
Similar for the left-up, right-down and right-up corner cells.

\begin{rem}
Since this example has rectangle grids, we can also define the flux $F$ in the $x$-direction and $y$-direction respectively
\begin{equation}
\begin{aligned}
&F_{i-\frac12, j} = F^L_{i,j}, \quad  i=2, \cdots, N, \,\, j=1, \cdots, M,\\
&F_{i, j-\frac12} = F^D_{i,j},\quad   i=1, \cdots, N, \,\, j=2, \cdots, M,\\
&F_{\frac12, j} = F_{N+\frac12, j}=0, \quad  j=1, \cdots, M, \qquad F_{i, \frac12} = F_{i,M+\frac12}=0, \quad  i=1, \cdots, N.
\end{aligned}
\end{equation}
Then \eqref{fp2d} can be recast as
\begin{equation}
\dot{\rho}_{i,j}\Delta x \Delta y  + F_{i+\frac12, j} - F_{i-\frac12,j} + F_{i, j+\frac12} - F_{i,j-\frac12} =0.
\end{equation}
\end{rem}

{ \section{Comparison between our upwind scheme with other finite volume schemes}\label{app:SG}
The famous Scharfetter–Gummel(SG) scheme was first proposed in \cite{scharfetter1969large} for 1D semiconductor device equation and there are many mathematical analysis and extensions on it, c.f., \cite{markowich1988inverse, xu1999monotone} and recently summarized as $B$-schemes in \cite{chainais2020large}.  We  remark it also has the $Q$-matrix structure and  the $\pi$-symmetric decomposition. Indeed,  the  finite volume scheme can be reformulated as
\begin{equation}
\begin{aligned}
\frac{\ud}{\ud t} \rho_i|C_i| = &\sum_j \frac{D|\Gamma_{ij}|}{|\my_j-\my_i|} \Big[ B\bbs{ \frac1D(\vec{b}\cdot \vec{n})_{ij}|\my_j-\my_i| } \rho_j - B\bbs{ \frac1D(\vec{b}\cdot \vec{n})_{ji}|\my_j-\my_i| }\rho_i  \Big];\\& \text{Scharfetter–Gummel:} \quad  B(x) = \frac{x}{e^x-1},\,\,x\neq0,\quad  B(0)=1;\\
& \text{Upwind scheme:} \quad  B(x) = 1+x^-.
\end{aligned}
\end{equation}
Since $B(x)> 0$, then with $Q_{ji}=\frac{D|\Gamma_{ij}|}{|\my_j-\my_i||C_j|}  B\bbs{\frac1D(\vec{b}\cdot \vec{n})_{ij}|\my_j-\my_i| }$, we have exactly same decomposition as \eqref{decom}. Specially, for the reversible case, let $\varphi$ be the potential such that $\vec{b}=-\nabla \varphi$ and thus $\pi\propto e^{-\varphi/D}$. Take the approximation
$(\vec{b}\cdot \vec{n})_{ij}|\my_j-\my_i| \approx \varphi_i - \varphi_j,$
then for $B(x)=\frac{x}{e^{x}-1}$, one can directly verify the detailed balance condition for $\pi =e^{-\varphi_i/D}$ 
$$Q_{ji}\pi_j|C_j|=\frac{|\Gamma_{ij}|}{|\my_j-\my_i|} \frac{\varphi_i -\varphi_j}{e^{\varphi_i/D}-e^{ \varphi_j/D}} = \frac{|\Gamma_{ij}|}{|\my_j-\my_i|} \bbs{\Lambda(\pi_i^{-1}, \pi_j^{-1})}^{-1} = Q_{ij}\pi_i|C_i|,$$
where $\Lambda(x,y)=\frac{x-y}{\log x -\log y}$ is the so-called logarithmic mean.  Hence in \eqref{decom}, $F_{ji}^\pi=0$ while
the symmetric coefficient $\alpha_{ij}$ in \eqref{decom} then becomes
\begin{align*}
\alpha_{ij} = Q_{ji}\pi_j|C_j| = \frac{D|\Gamma_{ij}|}{|\my_j-\my_i|}\bbs{ \int_0^1 \frac{1}{e^{-\varphi_i/D+s(\varphi_i - \varphi_j)/D}} \ud x }^{-1}, 
\end{align*}
where the average is the harmonic average of $\pi$ with  linear interpolation for $\varphi$ along edge $[\my_i,\my_j]$; c.f. \cite{markowich1988inverse, markowich1985stationary}. If replacing the harmonic average above by the arithmetic  mean of $\pi_i$ and $\pi_j$, then SG scheme becomes our  scheme \eqref{mp-pi} in the reversible case, i.e., $\vec{u}=0$ in \eqref{nb}.
}

\section{Remarks on Hill's derivation on the positive entropy production rate}\label{app:Hill}
We remark that \textsc{Hill} derives \eqref{entropyP} by calculating the change of free energy individually as follows. Denote $G_i$ as the Gibbs free energy at site $i$. Then the chemical potential at site $i$ at the steady state is
\begin{equation}\label{cp}
\mu_i = G_i+\mu_{M} + kT \log \pi_i.
\end{equation}
Here $\mu_{M}$ is the energy provided by the environment (heat bath in the open system) which  probably has different values for different orientations of a circulation in a biochemical reaction. For instance, in a simple Enzyme-substrate-product example \cite[Section 9]{hill2005free}, $\mu_M = \mu_S$ is the chemical potential of the substrate for a positive circulation, while  $\mu_M = \mu_P$ is the chemical potential of the product for an opposite circulation.
Then by the Arrhenius law, the transition rate from site $i$ to site $j$ is given by
\begin{equation}
\frac{Q_{ij}}{Q_{ji}}= e^{\frac{G_i- G_j + \mu_M}{kT}},  \quad \mu_M=\mu_S \text{ for } i=1,\,\,\, \mu_M = -\mu_P \text{ for }j=1.
\end{equation}
This together with \eqref{cp}, implies
\begin{equation}
\mu_i-\mu_j = kT \log \frac{Q_{ij} \pi_i}{Q_{ji}\pi_j}.
\end{equation}
Thus at the steady state, the individual energy change (in the unit of energy per unit time) due to the nonzero flux $i\to j$ is always
\begin{equation}
F^\pi_{ij} \bbs{\mu_i-\mu_j} = \bbs{Q_{ij}\pi_i - Q_{ji}\pi_j} kT \log \frac{Q_{ij} \pi_i}{Q_{ji}\pi_j} \geq 0.
\end{equation}
Adding each sites together and with our notation, we obtain \eqref{entropyP}.
{\blue Using Kolmogorov's circulation criterion, we can check
\begin{equation}
\frac{Q_{12}Q_{23}\cdots Q_{n1}}{Q_{21}Q_{32}\cdots Q_{1n}} = e^{\frac{\mu_s -\mu_p}{kT}}\neq 1.
\end{equation}
In fact, if $Q$ satisfies Kolmogorov's circulation criterion, then by construction proof, one can define a $\pi_i:= \frac{Q_{12}Q_{23}\cdots Q_{ki}}{Q_{21}Q_{32}\cdots Q_{ik}}$ and prove it is uniquely defined and satisfies detailed balance condition $Q_{ji}\pi_j = Q_{ij}\pi_i.$

\end{document}